\documentclass[12pt]{amsart}
\usepackage{amsmath,amsfonts,amssymb,comment} 
\usepackage{pdfpages}
\usepackage{amsbsy}
\usepackage{amscd}

\usepackage{amsthm}
\usepackage{bbold}
\usepackage{mathrsfs}
\usepackage{verbatim}
\usepackage [all]{xy}
\usepackage[top=1in,bottom=1in,left=1.25in,right=1.25in]{geometry}

\usepackage{tikz-cd}

\newcommand{\hide}[1]{}

\theoremstyle{plain}
\newtheorem{thm}{Theorem}[section]
\newtheorem{prop}[thm]{Proposition}
\newtheorem{claim}[thm]{Claim}

\newtheorem{cor}[thm]{Corollary}
\newtheorem{lem}[thm]{Lemma}

\newtheorem*{Cartan-Dieudonne}{Cartan-Dieudonne theorem {\rm (\cite [Chapter I, Theorem 7.1]{Lam})}}
\theoremstyle{definition}

\newtheorem{rem}[thm]{Remark}

\theoremstyle{remark}

\newcommand{\HH}{{\mathbb H}}

\newcommand{\CC}{{\mathbb C}}
\newcommand{\QQ}{{\mathbb Q}}
\newcommand{\RR}{{\mathbb R}}

\newcommand{\PP}{{\mathbb P}}

\newcommand{\rank}{{\rm rank}}

\title{A generalization of twistor lines  for complex tori}
\author{Nikolay Buskin}
\address{Department of Mathematics, University of California San Diego, 9500 Gilman Drive \# 0112, La Jolla, CA 92093-0112, USA}

\email{nvbuskin@gmail.com}

\begin{document}
\begin{abstract}
In this work we generalize the classical notion of a twistor line
in the period domain of compact complex tori studied in
\cite{Twistor-lines}. We  introduce two new types of lines, 
which are non-compact analytic curves in the period domain.
We  study the analytic properties of the compactifications of the curves,
preservation of  cohomology classes of type (1,1) along them and prove the twistor path connectivity
of the period domain by the curves of one of the new types.
\end{abstract}
\setcounter{tocdepth}{1}
\sloppy
\maketitle
\tableofcontents

\section{Introduction}
We call a manifold $M$ of a real dimension $4m$ a {\it hypercomplex} manifold, if  there exist 
(integrable) complex structures
$I,J,K$ on $M$ satisfying the quaternionic relations $$I^2=J^2=K^2=-Id, IJ=-JI=K.$$
The ordered triple $(I,J,K)$ is called a {\it hypercomplex structure} on $M$.

%Let $M$ be a manifold of real dimension $4m$. 
A Riemannian $4m$-manifold $M$ with a metric $g$ is called {\it hyperk\"ahler} with 
respect to $g$ (see \cite[p. 548]{Hitchin}), 
if there exist complex structures $I$, $J$ and $K$ on $M$, such that 
$I,J,K$ are covariantly constant and 
are isometries of the tangent bundle $TM$ with respect to $g$, satisfying
the above quaternionic relations.
%$$I^2=J^2=K^2=-1, IJ=-JI=K.$$ 
We call the  ordered triple $(I,J,K)$ of such complex structures {\it a hyperk\"ahler
structure on $M$ compatible with $g$}. 

Every hyperk\"ahler manifold $M$ naturally carries the  underlying hypercomplex structure
and is thus hypercomplex.
A hypercomplex 
%hyperk\"ahler 
 structure $(I,J,K)$ gives rise to a sphere $S^2$ of complex structures on $M$:
$$S^2=\{aI+bJ+cK| a^2+b^2+c^2=1\}.$$

%We call the family $\{(M,\lambda)| \lambda \in S^2\}  \rightarrow S^2$ a {\it twistor family over the twistor sphere $S^2$}.

The well known examples of compact hyperk\"ahler manifolds are compact complex tori and irreducible holomorphic symplectic manifolds (IHS manifolds).
We recall that an IHS manifold is a simply connected compact K\"ahler manifold $M$  with $H^0(M,{\Omega}_M^2)$ generated by
an everywhere non-degenerate holomorphic 2-form $\sigma$.

It is known that in the period domain of 
an IHS manifold any two periods can be connected
by a path of twistor lines, see Verbitsky, \cite{Verb-Torelli}, or its short exposition in \cite{Bourbaki}. %Moreover, such a path can be chosen generic, that
%is, the manifolds corresponding to the periods  at intersections of successive lines in the path
%have trivial Picard groups.  
The twistor path connectivity of each of the two connected components of the period domain
of complex tori was proved in \cite{Twistor-lines}. 

In the present paper we generalize the notion of twistor lines to include certain 
non-compact analytic curves in the period domain of complex tori and study the geometry of such curves,
in particular, the behavior at infinity, the path connectivity problem and the preservation of K\"ahler classes
along the curves.

Let us recall the construction of this period domain.
Let   $V_\RR$ be a real vector space of real dimension
$4n$.  A compact complex torus of complex dimension $2n$
considered as a real smooth manifold is the  quotient $A=V_\RR/\Gamma$ of $V_\RR$ by a lattice $\Gamma$
with the complex structure  given by an imaginary endomorphism $I\colon V_\RR \rightarrow V_\RR, I^2=-Id$.
Following \cite{Twistor-lines} we denote the period domain of compact complex tori of complex dimension $2n$ by $Compl$. It is the  set of imaginary endomorphisms of $V_\RR$ and is diffeomorphic to the orbit $^GI$, where 
$G=GL(V_\RR)=GL(4n,\RR)$ acts via the adjoint action, $^gI=g\cdot I:=gIg^{-1}$. This action 
naturally (pointwise) extends
to the action on the set of twistor lines $S\subset Compl$.
The period domain $Compl$ consists
of two connected components, corresponding, non-canonically, to the two connected components of
$G$. We have the embedding of $Compl$ into the Grassmanian $Gr(2n,V_\CC)$ of $2n$-dimensional
complex subspaces in $V_\CC=V_\RR\otimes \CC$ given by $Compl \ni I \mapsto (Id-iI)V_\RR \in Gr(2n,V_\CC)$, which maps $Compl$ biholomorphically onto an open subset of $Gr(2n,V_\CC)$. 
The
complement of this open set is the real-analytic locus $\mathcal L_\RR =\{U \in Gr(2n,V_\CC)|\, U\cap V_\RR \neq \{0\}\}$
of $2n$-dimensional complex subspaces in $V_\CC$ having nontrivial intersection with $V_\RR$.
This locus $\mathcal L_\RR$ is of real codimension 1 in $Gr(2n,V_\CC)$ and it cuts $Gr(2n,V_\CC)$ into two
pieces each of which is the corresponding component of $Compl$. 
%These connected
%components correspond (non-canonically) to the connected components $GL^+(V_\RR)$ and 
%$GL^-(V_\RR)$ of the group $G=GL(V_\RR)$ for the homogeneous space 
%representation $Compl={}^G\! I$.

%Further we will be dealing with 
%a fixed connected component of $Compl$ which we will denote also by $Compl$.

Introduce the following  4-dimensional real algebras,
$$\HH(\varepsilon)=\langle i,j\,|\, i^2=-1, j^2=\varepsilon, ij+ji=0\rangle, \varepsilon=-1,0,1.$$
The above introduced twistor spheres arise from embeddings of the algebra of quaternions
$\HH=\HH(-1) \hookrightarrow End\,V_\RR$. The non-compact analogs arise
from the embeddings $\HH(\varepsilon) \hookrightarrow End\,V_\RR$ for $\varepsilon=0,1$.
In the next subsection we explain why it is natural to consider these algebras along with 
$\HH$.

%\subsection{The geometry of non-compact twistor lines}

\subsection{Algebraic characterization of twistor lines in $Compl$}
\label{Algebraic-section}
In this subsection we give an algebraic characterization of the above introduced (compact) twistor lines, which is then
generalized in order to define the non-compact analogs of twistor lines.

Let $I,J,K=IJ \in End\, V_\RR$ be complex structures satisfying the quaternionic identities
and $$S=S(I,J)=\{aI+bJ+cK|a^2+b^2+c^2=1\}$$ be the corresponding twistor sphere. 
The basis $I,J,K$ of the space $\RR^3=\langle I,J,K\rangle \subset End\,V_\RR$ is orthonormal with respect to
the bilinear form $(u,v)=-\frac{1}{4n}Tr(uv),$ which is positively definite on $\langle I,J,K \rangle$. 
%here $4n=dim_\RR\,V_\RR$ 
%and $uv$ is the standard product in the algebra $End\,V_\RR$. 
 Moreover, 
$u,v \in \langle I,J,K\rangle$ anticommute if and only if $u\perp v$.
The sphere $S$ is the sphere of radius 1 in $\RR^3=\langle I,J,K\rangle$ centered at the origin.
Let $J_1\neq \pm J_2$ be some complex structures in $S$. Then the plane 
$\langle J_1,J_2\rangle_\RR \subset End\,V_\RR$ intersects $S$ in the
circle $S \cap \langle J_1,J_2\rangle_\RR$, which contains a complex structure that anticommutes with
$J_1$ (this can be seen by the means of the orthogonalization process applied to the pair $J_1,J_2$).
Thus there exists a real number $\alpha$ such that $\alpha J_1+J_2$ determines,
after a scalar normalization, a complex structure in $S$  anticommuting with $J_1$,
%(it belongs to the great circle $S \cap \langle J_1,J_2\rangle_\RR \subset S$, where $ \langle J_1,J_2\rangle_\RR
%\subset End\, V_\RR$ is the 2-plane spanned by $J_1$ and $J_2$),
that is (even without the normalization), $$(\alpha J_1+J_2)J_1+J_1(\alpha J_1+J_2)=0,$$
which results in the relation 
$$
J_1J_2+J_2J_1=2\alpha\cdot Id.
$$
The fact that $\alpha J_1+J_2$ is proportional to a complex structure
brings the following restriction on $\alpha$, 
\begin{equation}
\label{Eqn-imaginary}
(\alpha J_1+J_2)^2=a\cdot Id, a<0,
\end{equation}
so that $J=\frac{\alpha J_1+J_2}{\sqrt{-a}}$ is a complex structure and $J$ anticommutes with $I=J_1$.
The left side $(\alpha J_1+J_2)^2$ of Equation (\ref{Eqn-imaginary}) is equal to $$-(1+\alpha^2)Id+\alpha\cdot (J_1J_2+J_2J_1)=(\alpha^2-1)Id,$$
and thus the condition $a=\alpha^2-1<0$ is simply the condition $|\alpha|<1$.
Thus the necessary and sufficient condition that nonproportional $J_1,J_2$
belong to the same twistor sphere $S=S(I,J)$ is that there exists $\alpha \in \RR$ such that
\begin{equation}
J_1J_2+J_2J_1=2\alpha\cdot Id, \, |\alpha|<1.
\end{equation}

If we drop the restriction $|\alpha|<1$ then the complex structure operators 
$J_1,J_2$ satisfying $J_1^2=J_2^2=-Id,J_1J_2+J_2J_1=2\alpha Id$ generate 
the 4-dimensional algebra $\HH(1)\subset End\,V_\RR$ if $|\alpha|>1$ and the algebra $\HH(0)\subset End\,V_\RR$ if $|\alpha|=1$. Indeed, if $|\alpha|>1$ then the above calculations show that
$R=\frac{\alpha J_1+J_2}{\sqrt{\alpha^2-1}}$ satisfies $R^2=Id$ and anticommutes with 
$I=J_1$, so that $I$ and $R$ generate a subalgebra of $End\,V_\RR$ isomorphic to $\HH(1)$.
If $|\alpha|=1$ then $N=\alpha J_1+J_2$ is a nilpotent operator, $N^2=0$, and $N$ anticommutes
with $I=J_1$  so that $I$ and $N$ generate a subalgebra isomorphic to $\HH(0)$.

Using the above introduced bilinear form $(\cdot,\cdot)$  
%on the subalgebra of $End\,V_\RR$ generated by $J_1,J_2$, 
 we can universally express the above  process of obtaining $J_1,\alpha J_1+J_2$ from the
original $J_1,J_2$ in all three cases $\varepsilon=-1,0,1$, as
the orthogonalization process 
applied  to $J_1,J_2$.

The image of the set of imaginary units (that is, elements, whose square is equal $-1$) 
of the algebra $\HH(\varepsilon),\varepsilon=-1,0,1,$ under 
a faithful representation
$\HH(\varepsilon) \rightarrow  End\,V_\RR$ is a subset in $Compl$, which we call a {\it (generalized) twistor line
of type $\HH(\varepsilon)$}. Certainly, we need to justify extending the terminology to the cases $\varepsilon=0,1$
by checking that thus defined subsets are indeed complex submanifolds in $Compl$, which 
is a part of the statement of Theorem \ref{Theorem-generalized-twistor-lines}. 

%We will discuss the definitions together with their 
%compactifications in Section \ref{Generalized-twistor-lines}.
%In this subsection we formulate our results about the complex-analytic geometry of the 
%twistor lines in $Compl$ which describess their compactifications in $Gr(2n,V_\CC)\supset Compl$
%and their relationship with the loci $Compl_\Omega$ for $[\Omega] \in H^{1,1}(A,\RR)$,
%in particular for $[\Omega]$ K\"ahler.

\begin{comment}
First of all, let us start with the definition of the non-compact twistor lines.
Let $e_1,e_2\in Compl$ be non-proportional complex structures satisfying $e_1e_2+e_2e_1=2\alpha Id$.

We know that if $|\alpha|<1$ then $e_1,e_2$ span an imaginary sphere, which is a topological sphere $S^2\subset Compl$.
%However if we allow $|\alpha|\geqslant 1$ then $e_1,e_2$ do not span a sphere anymore.
Consider first the case $|\alpha|>1$. In this case $e_1e_2+e_2e_1=2\alpha$
and so $e_1(\alpha e_1+e_2)+(\alpha e_1+e_2)e_1=0$, but $(\alpha e_1+e_2)^2=\alpha^2-1>0$
and so $R=\pm \frac{\alpha e_1+e_2}{\sqrt{\alpha^2-1}}$ anticommutes with $e_1$ and $R^2=Id$.
Set for convenience $R=\frac{\alpha e_1+e_2}{\sqrt{\alpha^2-1}}$.
Denote $e_1$ by $I$. 
\end{comment}

%From now we will assume that we are dealing with 
%a faithful representation  $\HH(\varepsilon) \rightarrow End\, V_\RR$, understanding 
Consider first the case $\HH(1)\hookrightarrow End\,V_\RR$.  Denote the
images of $i$ and $j$ under the embedding as $I$ and $R$, in order to emphasize 
that $j$ acts as a reflection (or rotation by $\pi$) operator on $V_\RR$. Then 
$I^2=-Id,R^2=Id,IR+RI=0$.

Let us describe the complex structure operators contained in the 
image of $\HH(1)\hookrightarrow Compl$. All such are actually contained in the subspace 
$\RR^3=\langle I,R,IR\rangle \subset End\,V_\RR$, that is, identifying $\HH(1)$ with its image under the embedding, we may write 
$$\HH(1)\cap Compl=\langle I,R,IR\rangle \cap Compl.$$ 
The combination $xI+yR+zIR$ is a complex structure operator if and only if
$(xI+yR+zIR)^2=(-x^2+y^2+z^2)Id=-Id$, that is, $x^2-y^2-z^2=1$.
 Then the set $$S(I,R)=\{xI+yR+zIR\,|\,x^2-y^2-z^2=1\}$$
which is a two-sheeted hyperboloid consisting of complex structures, is
a generalized twistor line of the type $\HH(1)$. Note that as the two connected components $S(I,R)^+,S(I,R)^-$ of $S(I,R)$ contain respectively
$I$ and $-I$, they both must be contained in the same connected component $Compl$ of the period domain. Indeed, we know that given $I\in Compl$ we can find a complex structure  $J \in Compl$  anticommuting with $I$ (see, for example, \cite{Twistor-lines}) 
so that $-I=JIJ^{-1}$,
and, as $\det\,J=1$, both $\pm I$ belong to the same connected component of $Compl$
(or, arguing topologically, each $I$ together with $-I$ is contained in a compact connected twistor line $S=S(I,J)$). 

Next, let us consider the case $|\alpha|=1$, that is $\HH(0)\hookrightarrow End\,V_\RR$. In this case we denote 
the images of the generators $i$ and $j$ of $\HH(0)$ by $I$
and 
%$N=\alpha e_1+e_2$,
 $N$, so that $I^2=-Id, N^2=0, IN+NI=0$. 
Then the set $S(I,N)=\HH(0)\cap Compl=\langle I,N,IN\rangle \cap Compl$,
$$S(I,N)=\{\pm I+yN+zIN\,|\,y,z\in \RR\}=\RR^2\cup \RR^2,$$
is a generalized twistor line of type $\HH(0)$.
%Note that while the nilpotence and anticommutativity conditions  determine $N$ only up to a scalar multiple, the twistor line $S(I,N)$ does not depend on the choice of a particular scalar multiple $N$.
 Indeed, we have $(xI+yN+zIN)^2=-Id$ 
if and only if $x=\pm 1$.
Again, as the connected components $S(I,N)^+, S(I,N)^-$ of $S(I,N)$ contain $I$ and $-I$, we have that the whole $S(I,N)$
is contained in a connected component of the period domain.

%Further we will call a connected component of a generalized twistor line a 
%{\it connected generalized twistor line} (of the respective type) or, shorter, a {\it connected twistor line}.

As we said above, the group $G=GL(V_\RR)$ acts, via the adjoint action, on the set of generalized twistor lines,
${}^{g} S(I,J)=S({}^{g} I,{}^{g} J)$ for $g \in G$,  this action certainly preserves the type of the curves.

We denote the tangent cone at the point $p$ of a possibly singular complex manifold $M$ by 
$TC_pM$.
\subsection{The formulations of results}
\begin{comment}
We recall now that the complex-analytic connected manifold $Compl$ via mapping $I\mapsto (Id-iI)V_\RR \in Gr(2n,V_\CC)$ is identified with a connected component of the open subset
in $Gr(2n,V_\CC), V_\CC=V_\RR\otimes \CC$, whose complement is the 
 locus (a real analytic Schubert cycle) 
$\mathcal L_\RR=\{U\in Gr(2n,V_\CC)\,|\, 
U\cap V_\RR\neq\{0\}\}$ of (real) codimension 1.
\end{comment}

The following theorem gives an analytic description of the subsets $S(I,R)$ and $S(I,N)$ of $Compl$
and describes their behavior at the ``infinity'' $\mathcal L_\RR=Gr(2n,V_\CC)\setminus Compl$.
\begin{thm}
\label{Theorem-generalized-twistor-lines}
Each of the two sets $S(I,R)=S(I,R)^+\cup S(I,R)^-$ and $S(I,N)=S(I,N)^+\cup S(I,N)^-$ consists of two connected components, each of the components is a smooth complex 1-dimensional
submanifold in $Compl$, diffeomorphic to an open 2-disk. %that are naturally obtained by generalizing the notion of a twistor line in $Compl$.
Their analytic topology closures   $\overline{S(I,R)}$ and $\overline{S(I,N)}$ in the Grassmanian $Gr(2n,V_\CC) \supset Compl$ 
are complex-analytic curves.  
%The euclidean closure of each of the components $\overline{S(I,R)^{\pm}}$ is diffeomorphic to
%a closed 2-disk with the  boundary  
%the real-analytic circle $S^1=\overline{S(I,R)}\cap \mathcal L_\RR
%\subset \{U\in Gr(2n,V_\CC)\,|\, 
%\dim_\RR\,U\cap V_\RR=2n\}\subset \mathcal L_\RR$ at infinity.

 The curve $\overline{S(I,R)}$ is a $\mathbb{P}^1\subset Gr(2n,V_\CC)$, 
that intersects $\mathcal L_\RR$ along 
%which is tangent to $\mathcal L_\RR$ along 
 the real-analytic circle $S^1=\overline{S(I,R)}\cap \mathcal L_\RR
\subset \{U\in Gr(2n,V_\CC)\,|\, 
\dim_\RR\,U\cap V_\RR=2n\}\subset \mathcal L_\RR$,
and for every $p\in S^1$ 
we have $T_p\overline{S(I,R)} \cap TC_p\mathcal L_\RR=T_pS^1$.

 The curve $\overline{S(I,N)}$ consists of two connected components 
$\overline{S(I,N)}^{\pm}=\overline{S(I,N)^{\pm}}$, each of which is a
$\mathbb{P}^1\subset Gr(2n,V_\CC)$ with exactly one point $p_{\pm}=\overline{S(I,N)}^{\pm} \cap \mathcal L_\RR$ at infinity. The points $p_{\pm}$ are singular points
of $\mathcal L_\RR$  and the  tangent planes 
$T_{p_\pm}\overline{S(I,N)}$ intersect the respective tangent cones $TC_{p_\pm}\mathcal L_\RR$ trivially.
%these cusp points $p_+$
%and $p_-$ are the only points of $\overline{S(I,N)}$ in $\mathcal L_\RR$ and the tangent rays at cusps
%intersect $\mathcal L_\RR$ transversally.
\end{thm}

%Let $A$ be a complex torus of dimension $2n$ and $S$ be a connected generalized twistor line containing $the period of $A$. 
Let %$Hdg_S \subset H^{1,1}(A,\RR)$ 
 $Hdg_S=\{\Omega \in Hom(\wedge^2V_\RR,\RR)\,|\, \lambda^t\Omega\lambda=\Omega,\lambda\in S\}$ be the subspace of the alternating forms $\Omega$ on $V_\RR$ 
(where $\lambda$ and $\Omega$ are written in a fixed basis of $V_\RR$), and 
determining cohomology classes in $H^{1,1}((V_\RR/\Gamma,\lambda),\RR)$ for all $\lambda \in S$. 
%staying of Hodge type $(1,1)$
%along the line $S$. 
%, that is $Hdg_S\subset H^{1,1}((V_\RR/\Gamma,\lambda),\RR)$ for all  $\lambda \in S$ %or, what is the same, $Hdg_S=\{[\Omega] \in H^{1,1}(V_\RR/\Gamma,\RR)\,|\,\lambda^t\Omega\lambda=\Omega\}$.
%
Note that as the connected components of every generalized twistor line 
are central-symmetric to each other, that is, for every $\lambda\in S^+$
we have $-\lambda \in S^-$, the subspace $Hdg_S$ does not change, if in its definition 
we replace $S$ with its connected component. 

In \cite{Twistor-lines} we considered ``A toy example'' of a compact twistor line $S$ in the period domain of 
complex tori of dimension 2. 
%Let $A$ be a particular torus with its period in the twistor line $S$.
There we showed that the dimension of $Hdg_S$
 is 3 and $Hdg_S$ does not contain any K\"ahler classes,
or, in other words,  following the notations of \cite{Twistor-lines},
$S$ is not contained in any  locus $Compl_\Omega=\{I\in Compl\,|I^t\Omega I=\Omega\}\subset Compl$,
where the alternating 2-form $\Omega$ represents a K\"ahler class in $H^{1,1}(A,\RR)$ for a 
compact complex $2n$-torus $A$ .
%$\Omega$ and $I$ are written in a certain fixed basis of $V_\RR$.

One may ask if a generalized twistor line $S$ is contained in any K\"ahler locus $Compl_\Omega$
or not. Here we answer this question  for the general dimension case. 

\begin{thm}
\label{Kahler-theorem}
For any twistor line $S$ of the type $\HH(-1)$ we have $\dim\,Hdg_{S}=2n^2+n$ and $Hdg_{S}$ 
does not contain any K\"ahler classes. All representations
$\HH(-1) \rightarrow End\,V_\RR$ are equivalent, or, which is the same,  the adjoint action of 
$GL(V_\RR)$ on the set of compact twistor lines in $Compl$ is transitive. 

%Both components $S^+,S^-$ of a twistor line $S$ of the type
%$\HH(1)$ share the same  subspace $Hdg_S$ 
%$Hdg_S:=Hdg_{S^+}=Hdg_{S^-}$ 
For any twistor line $S$ of type $\HH(1)$ we have $\dim Hdg_S=2n^2+n$. 
The subspace $Hdg_S$ contains two disjoint open cones of K\"ahler classes, corresponding to each of
 the connected components of $S$. 
%$S^+$ and $S^-$.  
 All faithful representations
$\HH(1) \rightarrow End\,V_\RR$ are equivalent, or, which is the same,
the action of $GL(V_\RR)$
on the set of twistor lines of type $\HH(1)$ in $Compl$ is transitive. 

There are $n$ non-equivalent faithful representations $\HH(0) \rightarrow End\,V_\RR$
parametrized by integers $1\leqslant k \leqslant n$, equivalently, there are $n$ distinct 
%$Ad\,GL(V_\RR)$
 $GL(V_\RR)$-orbits of lines of the type $\HH(0)$ in $Compl$.
For any non-compact twistor line $S$  of the type $\HH(0)$ we have  
$\dim Hdg_S=k(k+1)+(2n-k)^2$, %$2(n^2+(n-k)^2)$, 
for the respective parameter $k$,
and $Hdg_S$ does not contain any K\"ahler classes for either of the two connected components
of $S$.
\end{thm}

Now let us get to the problem of the twistor path connectivity of $Compl$.
 We say that $Compl$ is {\it $\HH(\varepsilon)$-connected}
if any two point in the same connected component of $Compl$ can be connected by a path (a chain) of 
connected components of twistor lines of the type $\HH(\varepsilon)$.
The $\HH(-1)$-connectivity of $Compl$ was proved in \cite{Twistor-lines},
here we present a proof that $Compl$ is $\HH(1)$-connected.

In fact we will formulate and prove the
connectivity for both compact lines and non-compact lines of the type $\HH(1)$  in a uniform manner,
so that we include the compact case in the formulation of the following theorem.

\begin{thm}
\label{Connectivity-theorem}
The period domain $Compl$ is $\HH(\varepsilon)$-connected for $\varepsilon=-1,1$.
%In each connected component of $Compl$ any two points  can be connected by a path of lines of the type $S(I,N)$ with $rank\,N=2n$ ($k=n$).
\end{thm}

The idea of the proof is very similar to that in \cite{Twistor-lines}. 
The methods that we are using do not seem to directly apply to the problem of $\HH(0)$-connectivity,
so this problem remains open.

\begin{comment}
\begin{rem}
The method of the proof of  Theorem \ref{Connectivity-theorem} does not allow us to tell anything definite
about $\HH(0)$-connectivity of $Compl$ so that the question if such connectivity holds is open.
\end{rem}
\end{comment}
\begin{rem}
Note that Theorem \ref{Kahler-theorem} and \ref{Connectivity-theorem}
do not imply the $\HH(1)$-connectivity of the locus of polarized tori in $Compl$. 
Here these two theorems merely establish that the paths of non-compact lines along each of which
some  K\"ahler classes survive can be used for connecting points, so that in that regard these lines are not 
more specific than the compact lines, along which, by Theorem  \ref{Kahler-theorem} no K\"ahler classes survive.
\end{rem}

\section{Proof of Theorem \ref{Theorem-generalized-twistor-lines}}
\label{Generalized-twistor-lines}
\subsection{The complex-analytic structure} Let us start with checking first that $S(I,R)$ and $S(I,N)$ are complex-analytic 
submanifolds in $Compl$.
In order to see this we note that the tangent space to $S(I,R)$ at $p=xI+yR+zIR$
is $$T_pS(I,R)=\{uI+vR+wIR\,|\, xu-yv-zw=0\},$$
and, as the complex structure on $TCompl \subset TEnd\,V_\RR|_{Compl}$ 
%$TGr(2n,V_\CC)|_{Compl}$ 
 acts on $T_pCompl$ by  the left multiplication by $p\in Compl \subset End\,V_\RR$
(see, for example, \cite{Twistor-lines}), $$l_p\colon T_pCompl\rightarrow T_pCompl, a \mapsto p\cdot a,$$
it is enough to check that $l_p(T_pS(I,R))=T_pS(I,R)$, that is, 
that for $uI+vR+wIR \in T_pS(I,R)$ we have $l_p(uI+vR+wIR)=(xI+yR+zIR)(uI+vR+wIR)=
(-xu+yv+zw)Id+(zv-yw)I+(zu-xw)R+(xv-yu)IR=(zv-yw)I+(zu-xw)R+(xv-yu)IR$
is in $T_pS(I,R)$, which reduces to checking that
$x(zv-yw)-y(zu-xw)-z(xv-yu)=0$. Thus $S(I,R)$ is a complex analytic subset in $Compl$.

\begin{comment}
Next, let us consider the case $|\alpha|=1$. In this case we set $I=e_1$
and $N=\alpha e_1+e_2$, so that $I^2=-Id, N^2=0, IN+NI=0$.
Note that while the nilpotence and anticommutativity conditions  determine $N$ only up to a scalar multiple, the sphere $S(I,N)$ does not depend on the choice of a particular scalar multiple $N$.
Indeed, the combination $(xI+yN+zIN)^2=-x^2Id$ is an imaginary unit if and only if $x=\pm 1$.
Then $$S(I,N)=\{\pm I+yN+zIN\,|\,y,z\in \RR\}=\RR^2\cup \RR^2.$$
\end{comment}
Now, to verify the analyticity of $S(I,N)$  for an arbitrary $p=\pm I +yN+zIN \in S(I,N)$ we 
need to check the invariance  of the respective tangent space, $l_p(T_pS(I,N))=T_pS(I,N)$.
Like previously $(\pm I +yN+zIN)(vN+wIN)=\pm(-w N+v IN) \in T_pS(I,N)$, which shows the required
$l_p$-invariance of $T_pS(I,N)$. 
%and similarly for $p=-I+yN+zIN$.

\begin{comment}
These sets, $S(I,R)$ and $S(I,N)$, are non-connected non-compact complex 1-dimensional
submanifolds in $Compl$ that are naturally obtained by generalizing the notion of a twistor line in $Compl$.
From the viewpont of the Grassmanian $Gr(2n,V_\CC) \supset Compl$ 
the respective closures $\overline{S(I,R)}$ and $\overline{S(I,N)}$ are 
complex-analytic curves , the first is connected and smooth $\mathbb{P}^1$, and the second consists of two components each of which is a genus 0 curve with one cusp singularity.
\end{comment}
%\subsection{}
%with respect to the Pl\"ucker embedding, as well as the classical twistor spheres are. 
\medskip

Let us now study the points of the euclidean closures of $\overline{S(I,R)},\overline{S(I,N)}$ at the infinity
of the space $Compl$, that is, at the %real codimension 1 
 locus $\mathcal L_\RR$.
%=\{V\in Gr(2n,V_\RR \otimes \CC)\,|\, 
%V\cap V_\RR\neq\{0\}\}$.
Further we separately consider the cases of $S(I,R)$ and $S(I,N)$.

\medskip

%\subsection{\it The case of $S(I,R)$.}
{\it The case of $S(I,R)$.} 
\subsection{Points at infinity} Fix $xI+yR+zIR \in S(I,R)$. Up to change of coordinates, we can always assume that $z=0$ as $(yR+zIR)^2=y^2+z^2=x^2-1\geqslant 0$
and introducing $R_1=\frac{1}{\sqrt{x^2-1}}(yR+zIR)$ we get that $xI+yR+zIR=xI+\sqrt{x^2-1}R_1
\in S(I,R_1)=S(I,R)$. Next, we consider the real  curve 
$t\mapsto c(t)=tI+\sqrt{t^2-1}R \in S(I,R),t\geqslant 1$, this curve is actually contained in $S(I,R)^+$. 
The corresponding curve in the Grassmanian $C\colon [1,\infty) \rightarrow Gr(2n, V_\CC)$ is given by
$y \mapsto (Id-ic(t))V_\RR$. %Let us choose an $\RR$-basis $e_1,\dots, e_{2n},Ie_1,\dots, Ie_{2n}$
%of $V_\RR$. 
Then we have the equality of the points of the Grassmanian
$$C(t)=(Id-ic(t))V_\RR=\frac{1}{t}(Id-ic(t))V_\RR,$$
so instead of $Id-ic(t)$ we may consider the operator
$$\frac{1}{t}(Id-ic(t))=
\frac{1}{t}Id-i\left( I+\frac{\sqrt{t^2-1}}{t}R\right),$$ %and so $C(y)=\frac{1}{y}(Id-ic(y))V_\RR$, 
for which we see that $\underset{t \rightarrow \infty}{\lim }\, \frac{1}{t}(Id-ic(t)) =-i(I+R) \in End\, V_\CC$. 
%which does not determine a point in $G(2n,V_\CC)$ 

Let us clarify on what kind of an operator $I+R$ is. 
It is clear that, as $I$ and $R$ anticommute, $(I+R)^2=0$ and $I+R\neq 0$.
 Next, $I+R=I(Id-IR)=(Id+IR)I$ and we see that  $I$ provides an isomorphism 
between eigenspaces 
$Ker\,(Id-IR)$ and  $Ker\,(Id+IR)$ of $IR$, which together with $V_\RR=Ker\,(Id-IR) \oplus Ker\,(Id+IR)$
gives that $\dim\, Ker\,(Id-IR)=\dim\, Ker\,(Id+IR)=2n$.
%$Ker\, (Id-R)$ is mapped by $I$ into $Ker\,(Id+R)$
%and vice versa, so that  $\dim\,Ker\,(Id-R)=\dim\,Ker\,(Id-R)=2n$. Let $Ker\,(Id-R)=\langle e_1,\dots, e_{2n}\rangle$.
%Using the basis  $e_1, \dots,e_{2n}, Ie_1,\dots,Ie_{2n}$ and
So we conclude that $\rank\, (I+R)=\rank\,(Id-IR)=2n.$ As $(I+R)^2=0$ we see that
$Im\,(I+R)\subset Ker\, (I+R)$ and so, from 
$\dim\,V_\RR=\dim\, Ker(I+R)+\dim\,Im(I+R)=4n$, we must have $Im\,(I+R)= Ker\, (I+R)$.

The  real subspace $(I+R)V_\RR$
is obviously not a complex subspace in $V_\CC$.
In order to determine the actual limit $\underset{t \rightarrow \infty}{\lim }\, (Id-ic(t))V_\RR
=\underset{t \rightarrow \infty}{\lim }\, \frac{1}{t}(Id-ic(t))V_\RR$ in $Gr(2n,V_\CC)$, 
we need to use that $c(t)$ acts by multiplication by $i$ on $(Id-ic(t))V_\RR$. As
$\underset{t \rightarrow \infty}{\lim }\,\frac{1}{t}c(t)(Id-ic(t))=(I+R)$, we have that
$\underset{t \rightarrow \infty}{\lim }\,(Id-ic(t))V_\RR=(I+R)V_\RR\oplus i(I+R)V_\RR$,
which is a complex vector subspace of complex dimension $2n$ in $V_\CC$, this subspace belongs to $\mathcal L_\RR\subset Gr(2n,V_\CC)$.

Next, replacing $R$ with a general reflection operator $R_1$ as above, we get the whole circle
of limit points in $\mathcal L_\RR$,
$$S^1=\{(I+\cos t \cdot R+\sin t \cdot IR)V_\RR \oplus i(I+\cos t \cdot R+\sin t \cdot IR)V_\RR\,|\, t\in[0,2\pi]\},$$
which can be considered as the points of (a connected component of) $S(I,R)$ at infinity (it is easy to check
that distinct values of $t\in[0,2\pi)$ indeed correspond to distinct points in $\mathcal L_\RR$).

Now let us apply the central symmetry $xI+yR+zIR\mapsto -xI-yR-zIR$ 
to the curve $c(t) \subset S(I,R)^+$ so as to get the curve $-c(t)=-tI-\sqrt{t^2-1}R \subset S(I,R)^-$. %which now belongs to another connected component of $S(I,R)$. 
 Then, similarly to the previous, $\underset{t \rightarrow \infty}{\lim }\, \frac{1}{t}(Id+ic(t))V_\RR
=(I+R)V_\RR\oplus i(I+R)V_\RR$, that is, we get the same point in $S^1$ as for the previously considered
component of $S(I,R)$. This means that $\overline{S(I,R)}$ is connected (the two sheets of the hyperboloid glue together) and $S^1=\overline{S(I,R)}\cap \mathcal L_\RR$.

\begin{comment}
\begin{rem}
Although the embedding $Compl \hookrightarrow Gr(2n,V_\CC), I\mapsto (Id-iI)V_\RR$ is defined only on the points of $\langle I,R,IR \rangle_\RR \cap Compl$, one can see that it extends naturally to the
asymptotic cone (with the origin removed) of $S(I,R) \subset \langle I,R,IR \rangle_\RR$. 
Indeed, the cone with the origin removed is the union of the rays
$$\pm \RR^+\{I+\cos\, t\cdot R+\sin\,t\cdot IR\,| 0\leqslant t<2\pi\},$$
and the extension of our embedding sends each couple of rays 
$\pm\RR^+(I+\cos\, t\cdot R+\sin\,t\cdot IR)$ to the point $(I+\cos\, t\cdot R+\sin\,t\cdot IR)V_\RR\oplus i(I+\cos\, t\cdot R+\sin\,t\cdot IR)V_\RR$.
\end{rem}
\end{comment}

\subsection{The tangent cone at infinity: a simple example} 
Our goal is to prove the smoothness of the closure $\overline{S(I,R)}$ at points 
$p\in \overline{S(I,R)}\cap \mathcal L_\RR$. The study of the tangent cone
$TC_p\overline{S(I,R)}$ is done via evaluating tangent vectors at $p$ to  curves through $p$,
and, most conveniently, such curves are obtained by ``reversing'' the general curves 
$C(t)\subset Gr(2n,V_CC)$ with $\underset{t\to \infty}{\lim}C(t)=p$ so that we consider
them as starting at $p$. 

%Next, let us prove the smoothness of the closure $\overline{S(I,R)}$. 

% and that $\overline{S(I,R)}\subset Gr(2n,V_\CC)$ is tangent to $\mathcal L_\RR$, that is, at $p \in 
%\overline{S(I,R)}\cap \mathcal L_\RR$ we have 
%$T_p\overline{S(I,R)} \subset TC_p\mathcal L_\RR$. 
%To prove the tangency statement it is sufficient, therefore, 
%to show that 
%if we launch an arbitrary curve of the above considered form $(Id-ic(y))V_\RR$ from $p=\underset{y \rightarrow \infty}{\lim }\,
%(Id-ic(y))V_\RR=(I+R)V_\RR\oplus i(I+R)V_\RR$
%in the ``backward direction'', the (nonzero, for an appropriate parametrization) tangent vector to it 
%at $p$ will belong to $TC_p\mathcal L_\RR$.

%Let us start with proving the tangency statement first and then we will show the smoothness.
%but not to $T_p(\overline{S(I,R)}\cap \mathcal L_\RR)$. 
Let us first demonstrate this approach for  the simplest version of such a curve $C$, namely the one from the previous
 subsection. Of course, this curve will provide just one of vectors in $TC_p\overline{S(I,R)}$, but it will
then be  clear how to proceed with the most general curves. 
 
Replacing $t$ with $\frac{1}{t}$ 
we can rewrite the  ``reversed version'' of our curve $C$ as $$\widetilde{C}\colon t \mapsto \left(Id-ic \left (\frac{1}{t}\right)\right)V_\RR=
\left(Id-\frac{i}{t}\left( I+\sqrt{1-t^2}R\right)\right)V_\RR=$$ $$=\left(t\cdot Id-i
\left( I+\sqrt{1-t^2}R\right)\right)V_\RR, 0<t\leqslant 1.$$
again, we scale the curve of operators so that the limit  of this curve in $End\,V_\CC$ exists as $t\to 0$.
Then the tangent vector $\widetilde{C}^{\prime}(0) \in Hom(p,V_\CC/p)$ is found as follows.  
Let $U \subset V_\RR$ be a vector subspace such that $V_\RR=Ker\, (I+R) \oplus U$.
Then for every $v \in (I+R)V_\RR$ there is a unique $u\in U$ such that $v=(I+R)u$ and clearly 
$v=\underset{t\to 0}{\lim} \,i(t\cdot Id-i( I+\sqrt{1-t^2}R))u$, where
$i(t\cdot Id-i( I+\sqrt{1-t^2}R))u \in \widetilde{C}(t)$.
We launch a curve $\Phi_t\in Hom(p, \widetilde{C}(t))\subset Hom(p, V_\CC)$  sending $v\in (I+R)V_\RR\subset p$ to 
$$\Phi_t(v)=i(t\cdot Id-i( I+\sqrt{1-t^2}R))u \in \widetilde{C}(t)$$
and $iv\in i(I+R)V_\RR \subset p$ to 
$$\Phi_t(iv)=-(t\cdot Id-i( I+\sqrt{1-t^2}R))u \in \widetilde{C}(t),$$ $\Phi_0$ is the identity embedding
$p\hookrightarrow V_\CC$.  Let us differentiate $\Phi_t$ at $t=0$
and naturally descend the obtained homomorphism $\varphi=\left.\frac{d\Phi_t}{dt}\right|_{t=0} \in Hom(p,V_\CC)$ to a homomorphism $\widetilde{\varphi}$  in $Hom(p,V_\CC/p)\cong T_pGr(2n,V_\CC)$. 
The homomorphism $\varphi$ in $Hom (p, V_\CC)$
is the mapping 
$$v\in (I+R)V_\RR \mapsto \left.\frac{d}{dt}\left(i(t\cdot Id-i(I+\sqrt{1-t^2}R))u\right)\right|_{t=0}=iu,$$
$$iv\in i(I+R)V_\RR \mapsto \left.\frac{d}{dt}\left(-(t\cdot Id-i(I+\sqrt{1-t^2}R))u\right)\right|_{t=0}=-u.$$
%$$w\in (I+R)V_\RR \mapsto \left.\frac{d}{dt}\left(c\left(\frac{1}{t}\right)\cdot (t\cdot Id-i(\sqrt{1-t^2} I+R))w\right)\right|_{t=0}=iw.$$
%the latter, given that $c(\frac{1}{t})(tId-itI(\frac{1}{t}))=itId+tI(\frac{1}{t})=
%tiId+(\sqrt{1-t^2}I+R)$, is equal to $iv$. 
%Now set $J(t)=tId-i(I+R)$.
%Our tangent vector $\widetilde{C}^{\prime}(0)$
%is also tangent to the curve $t\mapsto J(t)V_\RR 
%+ iJ(t)V_\RR  \in \mathcal L_\RR$. The sum of subspaces is not direct but results in a 
%$2n$-dimensional complex subspace in $V_\CC$ that actually belongs to  $\mathcal L_\RR$ for every $t$.
%This can be seen as follows: consider the mapping $V_\RR \oplus V_\RR \rightarrow V_\RR\oplus iV_\RR,
%(u,v) \mapsto (J(t)u, iJ(t)v)$. Then 
%for nonzero $t$ the vector $(u,v)$ is in the kernel of this mapping if and only if  

%$t\mapsto J(t)\langle e_1,\dots, e_{2n}\rangle 
%\oplus iJ(t)\langle e_1,\dots, e_{2n}\rangle \in \mathcal L_\RR$, $J(t)=tId+\sqrt{1-t^2}(I+R)$
%where $J(t)=t(Id-(\frac{\sqrt{1-t^2}}{t}I+\frac{1}{t}R))$. Certainly,
%$\frac{1}{t}(\sqrt{1-t^2}I+R)$ is a reflection operator, anticommuting with
%the complex structure operator $\frac{1}{t}(I+\sqrt{1-t^2})$, and so, like above we have
% $\rank_\RR\,J(t)=2n$, justifying that the curve is indeed contained in $\mathcal L_\RR$.  

We have $\varphi((I+R)V_\RR\oplus i(I+R)V_\RR)=iU\oplus U$ % \in \mathcal L_\RR$
and $U\cap (I+R)V_\RR=\{0\}$,  
so that $\widetilde{\varphi}\in Hom(p,V_\CC/p)$ is a nonzero vector. 
In fact, the vector $\widetilde{\varphi}$
does not depend on the choice of our complement $U$: if $U_1$ is another complement,
 giving rise to the corresponding $\varphi_1$, 
then for $v\in (I+R)V_\RR$ we have  $\varphi_1(v)-\varphi(v)=i(u_1-u)\in iKer(I+R)=iIm\,(I+R)$
so that $\widetilde{\varphi}_1=\widetilde{\varphi}$ in $Hom(p,V_\CC/p)$. 
%
%At the same time a nonzero tangent vector in $T_p\mathcal L_\RR$ 
%considered as a homomorphism in $Hom(p,V_\CC/p)$ must map a nonzero subspace in 
%the real space $(I+R)V_\RR$ to $V_\RR$, which our $\widetilde{C}^{\prime}(0)$ does not do.

We claim, that thus defined $\widetilde{\varphi}$ does not belong to $TC_p\mathcal L_\RR$. Indeed, a vector
$\widetilde{\varphi} \in T_pGr(2n,V_\CC)$ belongs to $TC_p\mathcal L_\RR$ if and only if
$\varphi(p\cap V_\RR) \cap V_\RR\neq \{0\}$.
%Indeed, for example, the curve
% $$\left\{\langle v+t\varphi(v)\,|\,v \in p\rangle, t\in \RR \right\} \subset \mathcal L_\RR$$
%has $\widetilde{\varphi}$ as its tangent vector at $p$ ($t=0$). 
However, $\varphi(p\cap V_\RR)=iU$, so that  $\widetilde{\varphi} \notin TC_p\mathcal L_\RR$.

This proves that 
$\widetilde{C}^{\prime}(0)$, as calculated above, is correctly defined and does not belong to 
 $TC_p\mathcal L_\RR$. 
In the next subsection we will show that $TC_p\overline{S(I,R)}$ is actually a plane
$T_p\overline{S(I,R)}$ and that 
 $$T_p\overline{S(I,R)} \,\cap\, TC_p\mathcal L_\RR = T_pS^1,$$ where  $S^1$ is the circle at infinity that we defined earlier. 
%At the same time 
%$\widetilde{C}^{\prime}(0)$ clearly does not belong to $$T_p(\overline{S(I,R)}\cap \mathcal L_\RR)=
%\RR(\{(I+R)v\mapsto IRv\}\oplus\{i(I+R)v \mapsto iIRv\}),$$
%which finally proves the stated transversality.

\subsection{The tangent cone at infinity: the general case} 
Now we apply the ``reversion'' to more general curves $C$, which finally allows us to evaluate
the tangent cone $TC_p\overline{S(I,R)}$ and to show that it is actually a 2-plane in $T_pGr(2n,V_\CC)$,
thus proving that $\overline{S(I,R)}$ is smooth at every $p\in \overline{S(I,R)}\cap \mathcal L_\RR$. 
%The smoothness of $\overline{S(I,R)}$ at $p\in \overline{S(I,R)}\cap \mathcal L_\RR$
%is shown by checking that the tangent cone $TC_p\overline{S(I,R)}$ is a 2-plane.

% In order to determine $TC_p\overline{S(I,R)}$ we need to consider all possible curves through $p$ 
Let us fix $p\in \overline{S(I,R)}\cap \mathcal L_\RR$. 
Again, up to a change of coordinates we may assume that $p=(I+R)V_\RR \oplus i(I+R)V_\RR$.  
As earlier, we define the curves $C(t)\subset Gr(2n,V_\CC)$ 
 with $\underset{t\to \infty}{\lim}C(t)=p$ 
in terms of the curves $c(t)=x(t)I+y(t)R+z(t)IR\subset S(I,R)^+$,
%together with their central-symmetric image $t\mapsto \pm c(t)=\pm(x(t)I+y(t)R+z(t)IR) \in S(I,R)$
%$$t\mapsto c(t)=x(t)I+y(t)R+z(t)IR \in S(I,R)^+,$$ 
namely, as  the respective curves $t\mapsto C(t)=(Id- ic(t))V_\RR \subset Compl$ and calculate
the tangent vectors to the reversed versions of the latter at $p$.  %Setting $x(t)=\frac{1}{t}$

%Let us consider a general curve $t\mapsto (Id- ic(t))V_\RR$. 
For a general  curve $c\subset S(I,R)$ as above we may assume that $x(t)=t$, $t\geqslant 1$,
so that the $c(\frac{1}{t})=\widetilde{x}(t)I+\widetilde{y}(t)R+\widetilde{z}(t)IR$, where now $0<t\leqslant 1$, 
 $$\widetilde{x}(t)=\frac{1}{t}, \,\widetilde{y}(t)=\sqrt{(\widetilde{x}(t))^2-1}\cos\, g(t), \,
\widetilde{z}(t)=\sqrt{(\widetilde{x}(t))^2-1}\sin\, g(t),$$
and we may assume that the  function $g(t)$ is differentiable at $t=0$ and  $g(0)=0$.
%$\underset{t\rightarrow \infty}{\lim}g(t)=0$. 
Set 
$b=g^{\prime}(0)$.
%$\underset{t\rightarrow \infty}{\lim}\varphi(t)=0,\varphi^{\prime}(0)=b \in \RR$. 
%Now set $x(t)=\frac{1}{t}$. 
Then we get the reversed curve $\widetilde{C}(t)=(Id-ic(\frac{1}{t}))V_\RR$, 
$$ \widetilde{C}(t)=
\left(tId-i\left(I+\sqrt{1-t^2}\cos\, g(t)\cdot R+\sqrt{1-t^2}\sin\, g(t)\cdot IR\right)\right)V_\RR,$$ so that
this curve passes through our $p=(I+R)V_\RR\oplus i(I+R)V_\RR$ at $t=0$. 
%In order for a nonzero tangent vector at $t=0$ to exist we need to require that the limit
%$\underset{t\rightarrow \infty}{\lim}\varphi^{\prime}(t)$ exist. Denoting the limit by $b$
Arguing as previously we get an element of $Hom(p,V_\CC)$,  
 $$\left.\frac{d}{dt}\widetilde{C}(t)\right|_{t=0}=
\left\{\begin{array}{cc}(I+R)u \mapsto (b\cdot IR+i Id)u \\ i(I+R)u \mapsto (-Id+b\cdot iIR)u 
\end{array}\right\}.$$
All such tangent vectors and their positive scalar multiples injectively project onto
 an open half-plane $P_+$ in $TC_p\overline{S(I,R)}\subset T_pGr(2n,V_\CC)=Hom(p,V_\CC/p)$, 
%whose complement in $T_p\overline{S(I,R)}$ is
%the line $\RR\cdot \{(I+R)u \mapsto IRu, i(I+R)u \mapsto iIRu \}$, which is precisely
whose boundary in 
 $T_p\overline{S(I,R)}$ is
the image $\widetilde{l}$ of the line $$l=\RR\cdot \{(I+R)u \mapsto IRu, i(I+R)u \mapsto iIRu \} \subset Hom(p,V_\CC),$$
under the projection $Hom(p,V_\CC) \to Hom(p,V_\CC/p)$, this is precisely
the tangent line at $p$ to the circle $S^1=\overline{S(I,R)}\cap \mathcal L_\RR$. 
%This gives
%that %every vector in $T_p\overline{S(I,R)}$ that is a combination of the vectors $\widetilde{C}^{\prime}(0)$
%with coefficient 1
%and the above spanning vector of $T_p(\overline{S(I,R)}\cap \mathcal L_\RR)$
%with coefficient $b$ is indeed tangent to a certain curve through $p$ in $\overline{S(I,R)}$.
%On the opposite, 
% all such curves provide tangent vectors lying in
%the plane with a removed line $\langle \widetilde{C}^{\prime}(0), T_p(\overline{S(I,R)}\cap \mathcal L_\RR)\rangle_\RR\setminus T_p(\overline{S(I,R)}\cap \mathcal L_\RR)
%\subset T_p\overline{S(I,R)}$. 
Considering now the 
curves
$t\mapsto (Id+ic(t))V_\RR$, which arise from the curves of operators $-c(t)\subset S(I,R)^-$, central-symmetric to 
the previously considered curves $c(t)$,  and evaluating tangent vectors at $t=0$ and their positive scalar multiples
for such curves in $Gr(2n,V_\CC)$, we get another open half-plane plane $P_-$ in $TC_p\overline{S(I,R)}$, with the same boundary line and
which is central-symmetric to the previously obtained one, $P_-=-P_+$.  Now it is clear that
the tangent cone $TC_p\overline{S(I,R)}=P_+\cup \widetilde{l} \cup P_-=$
$$\left\langle \left\{\begin{array}{cc}(I+R)u \mapsto IRu \\  i(I+R)u \mapsto iIRu\end{array}\right\}
+Hom(p,p),
\left\{\begin{array}{cc}(I+R)u \mapsto iu \\  i(I+R)u \mapsto -u\end{array}\right\}+Hom(p,p) \right\rangle,$$
is a 2-plane, so that $\overline{S(I,R)}$ is  smooth, in the differentiable sense, at the point $p$ in $\overline{S(I,R)}\cap \mathcal L_\RR$, so that we may write $TC_p\overline{S(I,R)}=T_p\overline{S(I,R)}$.  It is easy to see, analogously to how it was done above,
that there are no vectors in $P^+,P^-$ that belong to $TC_p\mathcal L_\RR$, and clearly
$\widetilde{l}\subset TC_p\mathcal L_\RR$, so that we get the desired equality
$T_p\overline{S(I,R)} \,\cap\, TC_p\mathcal L_\RR = T_pS^1$. 
Similarly one shows smoothness and this intersection equality at all other points in  $\overline{S(I,R)}\cap \mathcal L_\RR$.

As the tangent plane $T_p\overline{S(I,R)}$ is the limit of tangent planes $T_qS(I,R)$ 
that are all invariant
under the  complex structure operator on $TCompl\subset TGr(2n,V_\CC)$, we conclude that $T_p\overline{S(I,R)}$
is also invariant under the complex structure operator, thus $\overline{S(I,R)}$
is indeed a smooth complex-analytic manifold.
This completes the proof of the statement of the Theorem regarding $S(I,R)$. %of the stated transversality.

\vspace*{0.5
cm}
\hspace*{2cm}
\input{S-I-R-corr-2.pic}
\vspace*{-11.5cm}
\begin{center}
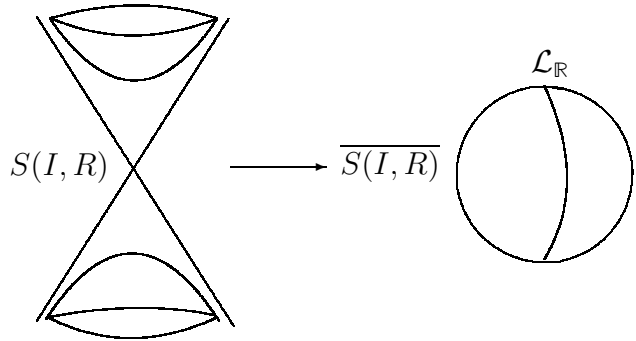
Picture 1: components $S(I,R)^{\pm}$ glue into a sphere, intersecting $\mathcal L_\RR$ in a circle.
\end{center}
\bigskip
{\it The case of $S(I,N)$.}

\subsection{Points at infinity} As $N^2=0$ we certainly know that $Im\, N \subset Ker\, N$, so that $\dim\,Im\, N \leqslant \dim\, Ker\, N$, and, as $\dim\,Im\,N+\dim\,Ker\,N=\dim\,V_\RR=4n$,
we must have $\rank\, N \leqslant 2n$, and in general it is possible to have a
strict inequality.

Now fixing $\alpha,\beta \in \RR$ and introducing $c(y)=I+\alpha yN+\beta y IN \in S(I,N)^+$, 
%where $S^+(I,N)$ is one of the two connected components of $S(I,N)$ 
%$c(y)=I-iI(y)$ 
we have the limit
$p_+=\underset{y\rightarrow \infty}{\lim}(Id-ic(y))V_\RR$, which is a (complex) $2n$-subspace in $V_\CC$,
containing the (possibly proper) complex subspace
$$(\alpha+\beta I)NV_\RR\oplus i (\alpha +\beta I)NV_\RR \subset p_+ \in \mathcal L_\RR,$$
where $\alpha$ acts as $\alpha Id$.
Besides that, if we look at the vector subspace $Ker\, N \subset V_\RR$, we 
can see that $(Id-ic(y))Ker\,N=(Id-iI)Ker\, N\subset V_\CC$, so that it is the same complex subspace in $V_\CC$
for all points $C(y)=(Id-ic(y))V_\RR$ of our curve in $Gr(2n,V_\CC)$.
Now given that $(\alpha +\beta I)N=N(\alpha -\beta I)$ and that for $\alpha, \beta$ not both zero
we have that $\alpha\pm\beta I$ is an invertible operator, and 
$(\alpha +\beta I)NV_\RR=N(\alpha -\beta I)V_\RR=NV_\RR$. Making analogous
calculations for  $S(I,N)^-$ with $p_-=\underset{y\rightarrow \infty}{\lim}(Id+ic(y))V_\RR$ we can finally write our $p_+$ 
and $p_-$ as
$$p_\pm=(Im\,N \oplus i\cdot  Im\,N) \oplus (Id\mp iI)Ker\, N.$$

 %Topologically the set of all such $p$ is a circle
%$\RR\mathbb{P}^1=\overline{S(I,N)}\cap \mathcal L_\RR$.

%Considering another connected component
So we see that $\overline{S(I,N)}$ consists of two connected components and 
$\overline{S(I,N)} \cap \mathcal L_\RR=\{p_+,p_-\}$.

\subsection{Tangent cone at infinity: reversing the curves} Again, launching the curve $C(y)$ from the point $p_+$ at infinity and setting $y=\frac{1}{t}$
we get the curve $$\widetilde{C}(t)=C\left(\frac{1}{t}\right)=(t\cdot Id-i(t\cdot I+\alpha N+ \beta IN))V_\RR.$$
%the tangent vector to this curve at $p$ is $\left.\frac{d}{dt}\{tV_\RR\}\right|_{t=0} \in 
%Hom()$.
In order to find the tangent vector to such a curve at $p_+$ ($t=0$) 
%$p\in \overline{S(I,N)} \cap \mathcal L_\RR$ ($t=0$),
as a homomorphism $\widetilde{\varphi}\in Hom(p_+,V_\CC/p_+)$
like in the  case of $S(I,R)$, we need to lift the curve $\widetilde{C}(t)$ to a curve $\Phi_t \in Hom(p_+,
\widetilde{C}(t))\subset Hom(p_+,V_\CC)$.
For that note that $Ker\,N \subset V_\RR$ is $I$-invariant and let us choose an $I$-invariant subspace $U\subset V_\RR$ such that $V_\RR=Ker\,N\oplus U$. Then
$N(\alpha -\beta I)\colon U \rightarrow Im\,N$ is an isomorphism and so for every $v \in Im\,N$ there exists
a unique $u\in U$ such that $v=N(\alpha -\beta I)u$.
%keep track of the vectors in $\widetilde{C}(t)$

% which is
%identified with the homomorphism in $Hom(p_+,V_\CC/p_+)$
Next, define $$\Phi_t(v)=i(t\cdot Id-i(t\cdot I+\alpha N+ \beta IN))u,$$
$$\Phi_t(iv)=-(t\cdot Id-i(t\cdot I+\alpha N+ \beta IN))u$$ for $v\in Im\,N$,
and $\Phi_t(v)=v$ for  $v \in (Id-iI)Ker\,N$. Then for $v\in Im\,N$ we have 
$\Phi_0(v)=(\alpha + \beta I)Nu=N(\alpha -\beta I)u=v$, so that $\Phi_0$
is the identity embedding $p_+\hookrightarrow V_\CC$.

\subsection{Tangent cone at infinity: evaluating the tangent vectors} Differentiating $\Phi_t$ at $t=0$ we get the vector $\varphi_{\alpha+i\beta} \in Hom(p_+,V_\CC)$ defined by 
$$v \mapsto i(Id-iI)u, iv \mapsto (-Id+iI)u \mbox{ for $v \in Im\, N$, } $$
$$(Id-iI)v \mapsto 0 %(Id-iI)(Id-iI)v=2(Id-iI)v
\mbox{ for $v \in Ker\, N$, }$$
 which naturally projects to a nonzero vector $\widetilde{\varphi}_{\alpha+i\beta} \in Hom(p_+,V_\CC/p_+)$. This vector is easily seen not to belong to $TC_{p_+}\mathcal L_\RR$. 
It is easy to check that the vectors $\varphi_{z_1},\varphi_{z_2}, z_1,z_2\in\CC\setminus\{0\},z_1\neq-z_2$, defined with respect to the same
$U$, satisfy 
\begin{equation}
\label{Addition-law}
\varphi_{z_1}+\varphi_{z_2}=\varphi_{\frac{z_1z_2}{z_1+z_2}}\mbox{ and }
a\varphi_z=\varphi_{\frac{z}{a}},z\in\CC\setminus\{0\},a\in\RR\setminus\{0\},
\end{equation}
 so that for an arbitrary $\alpha+i\beta \in \CC$ we have $\varphi_{\alpha+i\beta}=(\alpha^2+\beta^2)(\alpha\varphi_1+\beta\varphi_i)$.
Indeed, let us identify the subalgebra $\RR Id\oplus \RR\cdot I\subset End\,V_\RR$ with $\CC$ in an obvious way.  If $v=z_1Nu_1$ and $v=z_2Nu_2$ for $z_1,z_2\in \CC$, then $N(u_1+u_2)=\frac{1}{z_1}v+\frac{1}{z_2}v
=\frac{z_1+z_2}{z_1z_2}v$, so that $v=\frac{z_1z_2}{z_1+z_2}N(u_1+u_2)$ and this proves
the first of two equalities in (\ref{Addition-law}), the second equality is even easier.  

This tells us that our choice of $U$  provides a 2-plane of vectors $\varphi_z$ (and the zero vector) in $Hom(p_+,V_\CC)$. This plane injectively projects to a 2-plane of vectors $\widetilde{\varphi}_{\alpha+i\beta} \in Hom(p_+,V_\CC/p_+)\cong T_{p_+}Gr(2n,V_\CC)$. 
Again, similarly to the previous case,  these vectors $\widetilde{\varphi}_{\alpha+i\beta}$ do not
depend on the choice of the $I$-invariant complement $U$. 
%Varying $\alpha, \beta$ we get the whole tangent 
Thus $TC_{p^+}\overline{S(I,N)}$ is actually a 2-plane $T_{p_+}\overline{S(I,N)}$ and the
 intersection
$T_{p_+}\overline{S(I,N)}\cap TC_
{p_+}\mathcal L_\RR$ is zero. 

As $\dim_\RR \mathcal L_\RR=8n^2-1$, that is, it is a codimension 1 locus in $Compl$,
the tangent cone $TC_{p_+}\mathcal L_\RR$, intersecting the 2-plane $T_{p_+}\overline{S(I,N)}$
only at zero, cannot be an $8n^2-1$-dimensional vector subspace in the $8n^2$-dimensional
vector space $T_{p_+}Compl$. This means that $p_+$ is a singular point of $\mathcal L_\RR$. 
%The tangent cone of $\overline{S(I,N)}$ at the point $p_+ \in \overline{S(I,N)}\cap \mathcal L_\RR$
 %found analogously to the previous case.
\begin{comment}
 Indeed, approaching $p$ along the curves 
$c(y)V_\RR$ with $c(y)=Id-i(I+(\alpha y+\alpha_0)\cdot N+(\beta y +\beta_0)\cdot IN)$, where $\alpha_0,\beta_0 \in \RR$,
we get that $$\left.\frac{d}{dt}\{tc(y(t))V_\RR\}\right|_{t=0}=
\{v \mapsto (Id-i(I+\alpha_0 N+\beta_0 IN))v=(Id-iI)v,   $$ $$iv 
\mapsto i(Id-i(I+\alpha_0 N+\beta_0 IN))v=i(Id-iI)v \mbox { for $v \in Im\, N$} \},$$ and 
\end{comment}
%thus we get that the tangent cone to $\overline{S(I,N)}$
%at  $p \in \overline{S(I,N)}\cap \mathcal L_\RR$ 
%is a real ray, transversal to $\mathcal L_\RR$ at $p_+$, so $\overline{S(I,N)}$ is singular at $p_+$ 
All the above observations are also similarly true for  $p_-$.

\vspace*{0.5cm}
\hspace*{1.5cm}
\input{S-I-N-corr-3.pic}
\vspace*{-9.5cm}
\begin{center}
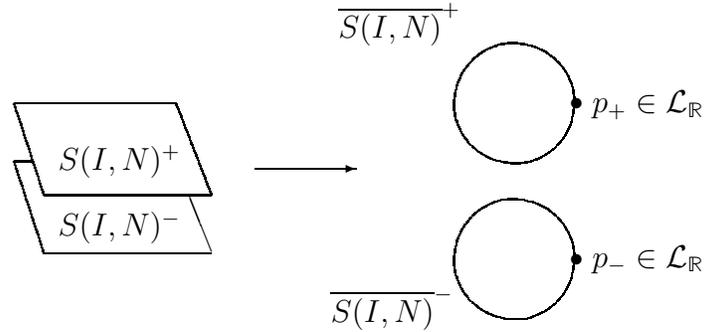
Picture 2: Each component of the closure $\overline{S(I,N)}$ has precisely one point in $\mathcal L_\RR$.
\end{center}

Now the proof of the theorem is complete.
%\end{proof}
%Moreover, considering now the other connected component of $S(I,N)$ we get that
%its point $p_- = (Im\,N \otimes\CC)\oplus \supset (Id+iI)Ker\, N$ at infinity is different from 
%$p_+=p=(Im\,N \otimes\CC)\oplus  (Id-iI)Ker\, N$, and so $\overline{S(I,N)}$
%consists of two cuspidal connected components.

\section{Proof of Theorem \ref{Kahler-theorem}}
%\subsection{Twistor lines in moduli of (pseudo-)polarized complex tori}

\subsection{} Further we will be using period matrices for verifying Riemann's bilinear relations, 
so we briefly recall the definition. Let $V$ be a $2n$-dimensional complex space,
$V\cong V_\RR$ over $\RR$, and let $I\colon V_\RR\rightarrow V_\RR$
corresponds to scalar multiplication by $i$ on $V$. Here instead of varying $I$
we will be varying the lattice $\Gamma=\langle \gamma_1,\dots,\gamma_{4n}\rangle \subset V_\RR$
determining our torus $V_\RR/\Gamma$. The period matrix, corresponding
to the complex torus $V_\RR/\Gamma$ with the complex structure $I$ 
is the matrix $(\gamma_1,\dots,\gamma_{4n})$ of vectors $\gamma_j\in V,j=1,\dots,4n$.
Choosing a basis $e_1,\dots,e_{2n}$ of $V$ allows us to write the period matrix as 
a $2n\times 4n$ matrix with complex entries. Next, assuming that
$\gamma_1,\dots, \gamma_{2n}$ are linearly  independent over $\CC$
and setting $e_1=\gamma_1,\dots,e_{2n}=\gamma_{2n}$ we can write our period matrix
in the form $(\mathbb{1}_{2n}|Z)$, where $\mathbb{1}_{2n}$ is the $2n\times 2n$-identity
matrix and $Z$ a $2n\times 2n$ matrix with complex entries. The matrix $(\mathbb{1}_{2n}|Z)$
is called a {\it normalized period matrix}. Let $I_Z$ be the matrix of the operator $I$
in the basis $\gamma_1,\dots,\gamma_{4n}$. Then we have the relation, coming from 
the definition of $I$, $$(\mathbb{1}_{2n}|Z)I_Z=(i\mathbb{1}_{2n}|iZ).$$ 

\subsection{} Next, for an alternating form $Q\in Hom(\wedge^2V_\RR,\RR)$ the first 
Riemann bilinear relation expresses a condition on $Z$, under which
$Q$ determines a class in $H^{1,1}(V_\RR/\Gamma,\RR)$ for the respective
lattice $\Gamma$. The second bilinear relation is a condition on $Z$,
under which $Q$ determines  a K\"ahler class in $H^{1,1}(V_\RR/\Gamma,\RR)$,
that is, $Q$ determines a positive hermitian form on $V$. We 
refer to  \cite[Ch. 2.6]{GrHarr} for the notations that we will use further.

\subsection{}
\label{Subsection-action} The natural action of $GL(V_\RR)$ on $V_\RR$ induces actions
on the alternating forms (and thus on second cohomology classes) and periods that agree: 
a form $Q$ determines a (positive) class in $H^{1,1}(V_\RR/\Gamma,\RR)$
if and only if $g^*Q=g^tQg$ determines a (positive) class in $H^{1,1}(V_\RR/
g^{-1}(\Gamma),\RR)$, $I_{Zg^{-1}}=gI_{Z}g^{-1}$, where we, by a slight abuse of notations,
identify the operator $g$ and its matrix in the fixed basis $\Gamma$. 

\medskip

%\begin{proof}[Proof of Theorem \ref{Kahler-theorem}]
{\it Case of a compact $S$.} All representations $\HH(-1)=\HH \rightarrow End\,V_\RR$
are equivalent and are direct sums of the unique irreducible 4-representations of $\HH$.
This implies the transitivity of the $GL(V_\RR)$-action on twistor lines of the type $\HH(-1)$
in $Compl$, discussed in \cite{Twistor-lines}. 

Now, the transitivity of the action on twistor lines and the remark made in subsection \ref{Subsection-action}
reduce the problem of counting the dimension of $Hdg_S$ and proving the absence
of K\"ahler classes in $Hdg_S$ to doing so just for some (arbitrary) particular twistor line $S\subset Compl$.  

So let us now choose a convenient twistor line $S$. 
Let $S=S(I,J)$ for anticommuting complex structures $I,J$, with matrices in the basis
$\Gamma$ as follows, 
%Let us fix a basis in $V_\RR$ 
%(not necessarily related to the basis  of the lattice $\Gamma$ defining the torus $A=\CC^{2n}/\Gamma$)
% such that the complex structures $I,J$ have the following matrices in this basis
\[I=
\left(\begin{array}{cc}
\mathbb{0}_{2n} & -\mathbb{1}_{2n}\\
\mathbb{1}_{2n} & \mathbb{0}_{2n}
\end{array}\right) \mbox{ and } J=
\left(\begin{array}{c|c}
I_{2n} & \mathbb{0}_{2n}\\
\hline
\mathbb{0}_{2n} & -I_{2n}\\
\end{array}\right), \mbox{ where }
I_{2n}=\left(\begin{array}{cc}
\mathbb{0}_{n} & -\mathbb{1}_{n}\\
\mathbb{1}_{n} & \mathbb{0}_{n}
\end{array}\right),
\]
$\mathbb{1}_k$ is the $k\times k$ identity matrix and $\mathbb{0}_k$ is 
the $k\times k$ zero matrix. 
Then, setting $K=IJ$,  we write the matrix of an arbitrary complex structure $\lambda=aI+bJ+cK \in S(I,J)$
in our basis   
as \[\lambda=
\left(\begin{array}{cc}
bI_{2n} & -a\mathbb{1}_{2n}+cI_{2n}\\
a\mathbb{1}_{2n}+cI_{2n} & -bI_{2n}
\end{array}\right), a^2+b^2+c^2=1.\]
For such a choice of basis we clearly have $\lambda^t=\lambda^{-1}=-\lambda$. 
The matrices $Q$ corresponding to the classes in $Hdg_S\subset H^{1,1}(V_\RR/\Gamma,\RR)$ 
%(where we assume that the compact torus $V_\RR/\Gamma$ is endowed with a complex structure)
must satisfy $\lambda^t Q \lambda=Q$ for all $\lambda \in S$,
or, what is the same, they must satisfy $QI=IQ, QJ=JQ$. Such $Q$'s will automatically satisfy the
first Riemann bilinear relation. 

Let \[Q=\left(\begin{array}{cc}
A & B \\
-B^t & D
\end{array}\right),\]
where $A,D$ are skew-symmetric $2n\times 2n$-matrices and $B$ an arbitrary $2n \times 2n$-matrix.
Then the commutation relation $QI=IQ$ means that $D=A$ and $B^t=B$. The commutation relation
$QJ=JQ$ means that if we write $A=
\left(\begin{array}{cc}
A_1 & A_2 \\
-A_2^t & A_4
\end{array}\right), \mbox{ where } A_1^t=-A_1,A_4^t=-A_4, B=\left(\begin{array}{cc}
B_1 & B_2 \\
B_2^t & B_4
\end{array}\right), \mbox{ where } B_1^t=B_1,B_4^t=B_4$, then as
$$JQ=\left(\begin{array}{c|c}
I_{2n} & \mathbb{0}_{2n}\\
\hline
\mathbb{0}_{2n} & -I_{2n}\\
\end{array}\right)\cdot \left(\begin{array}{cccc}
A_1 & A_2 & B_1 & B_2\\
-A_2^t & A_4 & B_2^t & B_4\\
-B_1 & -B_2 & A_1 & A_2\\
-B_2^t & -B_4 & -A_2^t & A_4
\end{array}\right)=\left(\begin{array}{cccc}
A_2^t & -A_4 & -B_2^t & -B_4\\
A_1 & A_2 & B_1 & B_2\\
-B_2^t & -B_4 & -A_2^t & A_4\\
B_1 & B_2 & -A_1 & -A_2
\end{array}\right),$$
and
$$QJ=
\left(\begin{array}{cccc}
A_1 & A_2 & B_1 & B_2\\
-A_2^t & A_4 & B_2^t & B_4\\
-B_1 & -B_2 & A_1 & A_2\\
-B_2^t & -B_4 & -A_2^t & A_4
\end{array}\right) \cdot \left(\begin{array}{c|c}
I_{2n} & \mathbb{0}_{2n}\\
\hline
\mathbb{0}_{2n} & -I_{2n}\\
\end{array}\right)=\left(\begin{array}{cccc}
A_2 & -A_1 & -B_2 & B_1\\
A_4 & A_2^t & -B_4 & B_2^t\\
-B_2 & B_1 & -A_2 & A_1\\
-B_4 & B_2^t & -A_4 & -A_2^t
\end{array}\right),$$
we have that $A_2^t=A_2, A_4=A_1, B_2^t=B_2, B_4=-B_1$.
Here we were able to get around without writing explicitly the first Riemann bilinear relation in the classical form, but in order to check the second one we need to write down the period matrix for a 
general point $\lambda \in S$.

The dimension of the space of matrices $Q$ satisfying the first bilinear relation is the 
sum of the dimensions of spaces of skew-symmetric $n\times n$-matrices $A_1$,
of symmetric $n\times n$-matrices $A_2,B_1,B_2$, that is $\frac{n(n-1)}{2}+3\cdot \frac{n(n+1)}{2}=
2n^2+n$. 
%=n^2+2\frac{n(n+1)}{2}=2n^2+n$. 
This is the dimension of the subspace $Hdg_{S(I,J)}$ of
classes  in $H^{1,1}(V_\RR/\Gamma,\RR)$ which stay of type $(1,1)$
along $S(I,J)$. 

Let us now find the respective normalized period matrices $\Omega=(\mathbb{1}_{2n}|Z)$
(not to be mixed with the 
notation $\Omega$ for the classes in $H^{1,1}(V_\RR/\Gamma,\RR)$ used in the introduction!) for the points of $S$,
$$\left(\mathbb{1}_{2n}|Z\right)\cdot \lambda =\left(i\mathbb{1}_{2n}|iZ\right),$$
so that
\[
\left(\begin{array}{cc}
\mathbb{0}_{n} & -b\mathbb{1}_{n}\\
b\mathbb{1}_{n} & \mathbb{0}_{n}
\end{array}\right)+Z\cdot
\left(\begin{array}{cc}
a\mathbb{1}_{n} & -c\mathbb{1}_{n}\\
c\mathbb{1}_{n} & a\mathbb{1}_{n}
\end{array}\right)=i\mathbb{1}_{2n},
\]
which gives 
$$Z=\frac{1}{a^2+c^2}\left(\begin{array}{cc}
(-bc+ai)\mathbb{1}_{n} & (ab+ci)\mathbb{1}_{n}\\
-(ab+ci)\mathbb{1}_{n} & (-bc+ai)\mathbb{1}_{n}
\end{array}\right).$$
The period matrices $\Omega=(\mathbb{1}_{2n}|Z)$ thus provide
an affine chart containing all of $S(I,J)$ except $\pm J$ (the points of $S$ for which  $a^2+c^2=0$).
Now following the Riemann bilinear notations in \cite[Ch. 2.6]{GrHarr}
we set $\widetilde{\Omega}=\left(\begin{array}{c} \Omega \\ \overline{\Omega}\end{array}\right)$,
where $\overline{\Omega}$ is obtained by complex conjugation of the entries of $\Omega$,
and we define the $4n\times 2n$-matrix $\widetilde{\Pi}=(\Pi ,\overline{\Pi})$, such that
$\widetilde{\Omega}\cdot \widetilde{\Pi}=\mathbb{1}_{4n}$. Then setting $\Pi=\left(\begin{array}{c} E \\ G\end{array}\right)$ we get the defining equations for $E$ and $G$,
$$E+ZG=\mathbb{1}_{2n}, \overline{E}+Z\overline{G}=\mathbb{0}_{2n}.$$
Subtracting from the first the complex conjugate of the second
we get that $(Z-\overline{Z})G=\mathbb{1}_{2n}$, and so 
\[Z-\overline{Z}=2Im\,Z=\frac{2i}{a^2+c^2}\left(\begin{array}{cc} 
a\mathbb{1}_{n} & c\mathbb{1}_{n} \\ -c\mathbb{1}_{n} & a\mathbb{1}_{n} \end{array}\right),
G=\frac{-i}{2}\left(\begin{array}{cc} 
a\mathbb{1}_{n} & -c\mathbb{1}_{n} \\ c\mathbb{1}_{n} & a\mathbb{1}_{n} \end{array}\right).
\]
and then $$E=\mathbb{1}_{2n}-\frac{1}{a^2+c^2}\left(\begin{array}{cc}
(-bc+ai)\mathbb{1}_{n} & (ab+ci)\mathbb{1}_{n}\\
-(ab+ci)\mathbb{1}_{n} & (-bc+ai)\mathbb{1}_{n}
\end{array}\right)\cdot \frac{-i}{2}\left(\begin{array}{cc} 
a\mathbb{1}_{n} & -c\mathbb{1}_{n} \\ c\mathbb{1}_{n} & a\mathbb{1}_{n} \end{array}\right)=$$
$$=\frac{1}{2}\left(\begin{array}{cc} 
\mathbb{1}_{n} & ib\mathbb{1}_{n} \\ -ib\mathbb{1}_{n} & \mathbb{1}_{n} \end{array}\right).$$

Now we are ready to verify if  the second Riemann bilinear relation $-i\Pi^tQ\overline{\Pi}>0$ holds. 
Scaling for convenience,  we write (where we replace entries of the kind $x\mathbb{1}_{n}$ with just $x$ to keep
the formulas compact)
$4\Pi^tQ\overline{\Pi}=$
\[
\left(\begin{array}{cccc} 1 & -bi & -ai & -ci\\
bi & 1 & ci & -ai\end{array}\right)
\left(\begin{array}{cccc}
A_1 & A_2 & B_1 & B_2\\
-A_2 & A_1 & B_2 & -B_1\\
-B_1 & -B_2 & A_1 & A_2\\
-B_2 & B_1 & -A_2 & A_1
\end{array}\right)
\left(\begin{array}{cc} 1 & -bi \\
 bi & 1 \\
ai & -ci \\
ci & ai \end{array}\right)=
\]
\[\small
\left(\begin{array}{cccc} 1 & -bi & -ai & -ci\\
bi & 1 & ci & -ai\end{array}\right)
\left(\begin{array}{cc} 
A_1+i(bA_2+aB_1+cB_2) & A_2+i(-bA_1-cB_1+aB_2) \\
 -A_2+i(bA_1+aB_2-cB_1) & A_1+i(bA_2-cB_2-aB_1) \\
-B_1+i(-bB_2+aA_1+cA_2) & -B_2+i(bB_1-cA_1+aA_2) \\
-B_2+i(bB_1-aA_2+cA_1) & B_1+i(bB_2+cA_2+aA_1) \end{array}\right)=
%\left(\begin{array}{cccc} 
%A_1+i(bA_2+aB_1+cB_2) & A_2+i(-bA_1+aB_2-cB_1) & B_1+i(-bB_2-aA_1+cA_2)& B_2+i(bB_1-aA_2-cA_1)
%\\
%-A_2+i(bA_1-cB_1+aB_2) & A_1+i(bA_2-cB_2-aB_1) & B_2+i(bB_1+cA_1+aA_2) & -B_1+i(bB_2+cA_2-aA_1)
%\end{array}\right)
%\left(\begin{array}{cc} 
%1 & -ib \\
% ib & 1 \\
%ai & -ci \\
%ci & ai 
%\end{array}\right)
\]
\[=2
\left(\begin{array}{cc} 
A_1+i(bA_2+aB_1+cB_2) & A_2+i(-bA_1-cB_1+aB_2) \\
 -A_2+i(bA_1-cB_1+aB_2) & A_1+i(bA_2-aB_1-cB_2) \\
 \end{array}\right).
\]

Now let us consider the upper-left $2n\times 2n$ block on the diagonal of the hermitian matrix $-i\Pi^tQ\overline{\Pi}$, 
which is the hermitian matrix $\frac{1}{2}((bA_2+aB_1+cB_2)-iA_1)$. This matrix is positively or negatively
definite if and only if its complex conjugate is such, so that if it is definite, then the real part
$\frac{1}{2}(bA_2+aB_1+cB_2)$ is definite as well. For every $\lambda=aI+bJ+cK\in S(I,J)$
we have that $-\lambda=-aI-bJ-cK \in S$, and the matrices $bA_2+aB_1+cB_2$
and $-bA_2-aB_1-cB_2$
cannot be both positively (or negatively) definite. % for all $\lambda\in S$. 
This means that among classes in $Hdg_{S(I,J)}$ there are no  K\"ahler classes.
 Thus, for a K\"ahler class $\Omega\in H^{1,1}(V_\RR/\Gamma,\RR)$  the complex submanifold $Compl_\Omega \subset Compl$ 
does not contain $S$, and so $Compl_\Omega$ and $S$ may only have  finitely many points in common.

\medskip

{\it Case of $S=S(I,R)$.} Like in the previous case, we first establish the transitivity
of the $GL(V_\RR)$-action on twistor lines of the type $\HH(1)$, which, together
with \ref{Subsection-action}, will then allow us to calculate $\dim\, Hdg_S$ and 
determine the K\"ahler classes in $ Hdg_S$ for a particular convenient line $S=S(I,R)$,
 extending then these results via the $GL(V_\RR)$-action on all lines of the type $\HH(1)$.

The anticommutation $IR=-RI$ tells us that
$I$ establishes an isomorphism between the eigenspaces $Ker\, (R-Id)$
and $Ker\, (R+Id)$ of the operator $R$. We have the decomposition $V_\RR=Ker\, (R-Id)\oplus Ker\, (R+Id)$.  
Let $v_1,\dots, v_{2n}$ be a basis of $Ker\, (R-Id)$,
then $Iv_1,\dots, Iv_{2n}$ is a basis of $Ker\, (R+Id)$. Then
we can write  the matrices of $I$
and $R$
in the basis  $v_1,\dots, v_{2n},Iv_1,\dots, Iv_{2n}$, 
\[
I=\left(\begin{array}{cc}
\mathbb{0}_{2n} & \mathbb{-1}_{2n}\\
\mathbb{1}_{2n} & \mathbb{0}_{2n}\\
\end{array}\right),
R=\left(\begin{array}{cc}
\mathbb{1}_{2n} & \mathbb{0}_{2n}\\
\mathbb{0}_{2n} & \mathbb{-1}_{2n}\\
\end{array}\right).
\]
This shows that all representations of the real 4-dimensional algebra $\HH(1)=\langle i,r|i^2=-1, r^2=1, ir+ri=0\rangle \rightarrow End\, V_\RR$ 
are equivalent, so that the group $GL(V_\RR)$ acts transitively on all twistor lines 
$S(I,R) \subset Compl \subset  End\,V_\RR$. In fact the decomposition 
$V_\RR=\langle v_1+Iv_1,v_1-Iv_1\rangle \oplus \dots \oplus \langle v_{2n}+Iv_{2n},v_{2n}-Iv_{2n}\rangle$ breaks down the natural representation $i \mapsto I,r \mapsto R \in End\, V_\RR$
into the sum of irreducible 2-representations, and any two 2-representations of our algebra are isomorphic
and faithful.

We now set $\Gamma=\langle v_1,\dots, v_{2n},Iv_1,\dots, Iv_{2n} \rangle$
and consider the particular line $S=S(I,R)$. 
Like in the previous case, an alternating form $Q$ determines a class in $Hdg_{S(I,R)}$ or,
which is the same, the first bilinear relation for $Q$ holds over $S(I,R)$, if and only if the relations $\lambda^tQ\lambda=Q$ hold for all $\lambda \in S(I,R)$. 
Note that, unlike in the previous case, the second generator $R$ of our algebra
$\HH(1)$ is not a complex structure, $R \notin S(I,R)$, and in the current case 
the relations $\lambda^tQ\lambda=Q,
\lambda \in S(I,R)$, are equivalent to the relations $IQ=QI,RQ=-RQ$. Indeed, 
setting $\lambda=I$ we get $IQ=QI$ and setting $\lambda=\sqrt{2}I+R$
we get the relation $(\sqrt{2}I+R)^tQ(\sqrt{2}I+R)=Q$. The latter, given
that $(\sqrt{2}I+R)^t=-\sqrt{2}I+R=(\sqrt{2}I-R)^{-1}$, is equivalent to
$Q(\sqrt{2}I+R)=(\sqrt{2}I-R)Q$, which, together with the $I$-invariance of $Q$, 
implies that $QR=-RQ$. On the opposite, if $Q$ is $I$-invariant and $R$-antiinvariant,
then $Q$ is $R_1$-antiinvariant for every $R_1=bR+cIR$, where $b^2+c^2=1$, and hence $\lambda$-invariant for every $\lambda=aI+\sqrt{a^2-1}R_1 \in S(I,R)$. 
%for every $\lambda=aI+\frac{\sqrt{a^2-1}}{\sqrt{b^2+c^2}}R_1 \in S(I,R)$.

The relations $IQ=QI,RQ=-QR$
tell us that the matrix $Q$ is an arbitrary matrix of the form
\[Q=
\left(\begin{array}{cc}\mathbb{0}_{2n} & B\\
-B & \mathbb{0}_{2n} \end{array}\right),\]
where $B$ is any $2n\times 2n$ real symmetric matrix.
%which we are going to show
%by writing down the first bilinear relation for $Q$ and all $\lambda \in S(I,R)$.  
%So let us 
%fairly write down the first bilinear relation for the points $\lambda \in S(I,R)$.

Now we need to write down the period matrices for the points $\lambda \in S(I,R)$
in order to formulate the second bilinear relation for $Q$. 
The matrix of the general $\lambda$ is given by
\[
S(I,R)\ni \lambda=aI+bR+cIR=
\left(\begin{array}{cc}
b\mathbb{1}_{2n} & (c-a)\mathbb{1}_{2n}\\
(a+c)\mathbb{1}_{2n} & -b\mathbb{1}_{2n}\\
\end{array}\right), a^2-b^2-c^2=1,
\]
and we write down the relations defining the period matrices
of points $\lambda \in S(I,R)$, setting as earlier $\Omega=(\mathbb{1}_{2n}|Z)$,
\[
(\mathbb{1}_{2n}|Z)\cdot 
\left(\begin{array}{cc}
b\mathbb{1}_{2n} & (c-a)\mathbb{1}_{2n}\\
(a+c)\mathbb{1}_{2n} & -b\mathbb{1}_{2n}\\
\end{array}\right)=(i\mathbb{1}_{2n}|iZ).
\]
This gives us $Z=\frac{-b+i}{a+c}\mathbb{1}_{2n}$, and as $a+c\neq 0$ for $a^2-b^2-c^2=1$,
our affine chart of matrices $\Omega$ contains all of $S(I,R)$.
As earlier we write $\widetilde{\Omega}=\left(\begin{array}{c} \Omega \\\overline{\Omega} \end{array}\right)$ and $\widetilde{\Pi}=(\Pi,\overline{\Pi})$ where $\widetilde{\Omega}\cdot \widetilde{\Pi}=\mathbb{1}_{4n}$. Setting $\Pi=\left(\begin{array}{c} E \\ G \end{array}\right)$ and denoting $u=\frac{-b+i}{a+c}$
we get 
\[
\widetilde{\Omega}\cdot \widetilde{\Pi}=
\left(\begin{array}{cc}
\mathbb{1}_{2n} & u\mathbb{1}_{2n}\\
\mathbb{1}_{2n} & \overline{u}\mathbb{1}_{2n}\\
\end{array}\right)\cdot
\left(\begin{array}{cc}
E & \overline{E}\\
G & \overline{G}\\
\end{array}\right)=\mathbb{1}_{4n}.
\]
Solving the latter matrix equation gives $G=\frac{1}{u-\overline{u}}\mathbb{1}_{2n}=-\frac{a+c}{2}i\mathbb{1}_{2n}$ and $E=\mathbb{1}_{2n}-uG=\mathbb{1}_{2n}-\frac{-b+i}{a+c}(-\frac{a+c}{2}i)
\mathbb{1}_{2n}=(1+\frac{-b+i}{2}i)\mathbb{1}_{2n}=\frac{1-bi}{2}\mathbb{1}_{2n}$.
\begin{comment}
Setting now $u_1=\frac{1-bi}{2},u_2=-\frac{a+c}{2}i$ we write the first bilinear relation as
$\Pi^tQ\Pi=$
\[
\left(u_1\mathbb{1}_{2n}, u_2\mathbb{1}_{2n}\right)
\left(\begin{array}{cc}A & B\\
-B^t & D \end{array}\right)
\left(\begin{array}{c}u_1\mathbb{1}_{2n}\\ u_2\mathbb{1}_{2n}\end{array}\right)=
(u_1A-u_2B^t,  u_1B+u_2D)
\left(\begin{array}{c}u_1\mathbb{1}_{2n}\\ u_2\mathbb{1}_{2n}\end{array}\right)=\]
\[
=u_1^2A+u_1u_2(B-B^t)+u_2^2D=\mathbb{0}_{2n}.
\]
Now for $a=1,b=c=0$ we get $u_1=\frac{1}{2},u_2=-\frac{i}{2}$,
so that $u_1^2A+u_1u_2(B-B^t)+u_2^2D=\frac{1}{4}((A-D)-i(B-B^t))=\mathbb{0}_{2n}$,
which is equivalent to $A=D, B=B^t$. Keeping in mind the latter equalities and setting $a=\sqrt{2},b=1,c=0$ so that $u_1=\frac{1-i}{2},
u_2=-\frac{\sqrt{2}}{2}i$ we get $u_1^2A+u_1u_2(B-B^t)+u_2^2D=(u_1^2+u_2^2)A
=\frac{1}{4}((1-i)^2+(-\sqrt{2}i)^2))A
=\frac{1}{4}(-2-2i)A=\mathbb{0}_{2n}$ we finally get the defining relations for the components of $Q$,$A=D=\mathbb{0}_{2n},B=B^t$. These relations for $Q$ are clearly equivalent to the relations
$IQ=QI$ and $RQ=-QR$, and they imply the relations $\lambda^t Q \lambda = Q$ for general $\lambda \in S(I,R)$.
\end{comment}

Setting now $u_1=\frac{1-bi}{2},u_2=-\frac{a+c}{2}i$, we write the second bilinear relation as $-i\Pi^tQ\overline{\Pi}=$
\[
-i\left(u_1\mathbb{1}_{2n}, u_2\mathbb{1}_{2n}\right)
\left(\begin{array}{cc}\mathbb{0}_{2n} & B\\
-B & \mathbb{0}_{2n} \end{array}\right)
\left(\begin{array}{c}\overline{u_1}\mathbb{1}_{2n}\\ \overline{u_2}\mathbb{1}_{2n}\end{array}\right)=
-i(-u_2B,  u_1B)
\left(\begin{array}{c}\overline{u_1}\mathbb{1}_{2n}\\ \overline{u_2}\mathbb{1}_{2n}\end{array}\right)=
\]
\[=-i(u_1\overline{u_2}-\overline{u_1}u_2)B=-\frac{i}{4}((1-bi)(a+c)i-(1+bi)(-(a+c)i))B=\]
\[
=\frac{a+c}{4}((1-bi)+(1+bi))B=\frac{(a+c)}{2}B>0.\]
If we restrict to the upper sheet $a=\sqrt{1+b^2+c^2}$ of our hyperboloid, we get that $B$ can be any positive definite matrix, and if we consider the lower sheet $a=-\sqrt{1+b^2+c^2}$ then $B$ can be any negative definite matrix in order for the relation $-i\Pi^tQ\overline{\Pi}>0$ to hold.
\begin{rem}
Note that while $\overline{S(I,R)}\subset Gr(2n,V_\CC)$ is a connected curve, the components
$S(I,R)^+,S(I,R)^-\subset Compl$ do not share any K\"ahler class staying of type (1,1) along
both of them.
\end{rem}

The dimension of all matrices $Q$ satisfying the first bilinear relation is the dimension of 
symmetric $2n\times 2n$-matrices $B$, that is, $\frac{2n(2n+1)}{2}=n(2n+1)$.
This is the dimension of the subspace of
classes $Hdg_{S(I,R)}$ in $H^{1,1}(V_\RR/\Gamma,\RR)$ which stay of type $(1,1)$
along $S(I,R)$. The K\"ahler classes which stay of type $(1,1)$
along one of the components of $S(I,R)$ correspond to $\pm B>0$ and form a subset
in  $Hdg_{S(I,R)}$ having two connected components, each of which is
an open cone in $Hdg_{S(I,R)}$. 

\medskip

{\it Case of $S(I,N)$.} 
Again, we start with the representation theory of the real algebra 
$\HH(0)=\langle i,n|i^2=-1,n^2=0,in+ni=0\rangle$.
%(here we are using more commonly acceptible notation $y$ for our nilpotent generator instead of
%using the controversial $n$).

As $IN=-IN$, we have that both the kernel and the image of $N$, $Im\,N\subset Ker\, N \subset V_\RR$
are $I$-invariant subspaces. Next, choosing an $I$-invariant complement $U$ for $Ker\, N$ in $V_\RR$,
$V_\RR=Ker\, N \oplus U$, a basis $u_1,\dots, u_k,Iu_1,\dots, Iu_k$ of $U$, and an $I$-invariant
complement $W$ for $Im\,N$ in $Ker\,N$ together with its basis of the form $w_1,\dots, w_l,Iw_1,\dots,
Iw_l$, we get the basis $Nu_1,\dots, Nu_k,NIu_1,\dots, NIu_k$ of $Im\,N$, and
we get the basis $Nu_1,\dots, Nu_k,NIu_1,\dots, NIu_k, w_1,\dots, w_l,Iw_1,\dots,
Iw_l, u_1,\dots, u_k,Iu_1,\dots, Iu_k$ of $V_\RR=Im\,N\oplus W\oplus U$, where $4k+2l=dim\,V_\RR=4n$.

In this basis $I$ and $N$ have the matrices
\[
I=\left(\begin{array}{c|c|c}
\begin{array}{cc}
\mathbb{0}_k & \mathbb{1}_k\\
-\mathbb{1}_k & \mathbb{0}_k\\
\end{array} & \mathbb{0}_{2k\times 2l} & \mathbb{0}_{2k}\\
\hline
\mathbb{0}_{2l \times 2k} & \begin{array}{cc}
\mathbb{0}_l & -\mathbb{1}_l\\
\mathbb{1}_l & \mathbb{0}_l\\
\end{array} &\mathbb{0}_{2l\times 2k}\\
\hline
\mathbb{0}_{2k} & \mathbb{0}_{2k\times 2l} & \begin{array}{cc}
\mathbb{0}_k & -\mathbb{1}_k\\
\mathbb{1}_k & \mathbb{0}_k\\
\end{array}
\end{array}\right),
N=\left(\begin{array}{c|c|c}
\mathbb{0}_{2k}  & \mathbb{0}_{2k\times 2l} & \mathbb{1}_{2k} \\
\hline
\mathbb{0}_{2l \times 2k} & \mathbb{0}_{2l} &\mathbb{0}_{2l\times 2k}\\
\hline
\mathbb{0}_{2k} & \mathbb{0}_{2k\times 2l} & \mathbb{0}_{2k}
\end{array}\right),
\]
which proves that the every representation of our algebra is
reducible (but not completely reducible). All 
non-equivalent $4n$-representations are parametrized by the values of 
nonnegative integers $k,l$ satisfying $4k+2l=4n$, that is, 
as $l=2(n-k)$, we have that there are $n$ such $4n$-representations.

Now let $Q$ be a skew-symmetric $4n\times 4n$-matrix written in our basis.
The first bilinear relation for $Q$ with respect to the periods in $S(I,N)$
can be written as $$(I+aN+bIN)^tQ(I+aN+bIN)=Q \mbox{ for all } a,b \in\RR.$$
Expanding the left side as a (matrix-valued) polynomial in $a,b$, $(I+aN+bIN)^tQ(I+aN+bIN)=I^tQI+
a(N^tQI+I^tQN)+b((IN)^tQI+I^tQ(IN))+
a^2N^tQN+b^2(IN)^tQ(IN)
+ab(N^tQIN+(IN)^tQN)=Q$ and equating the coefficients of the monomials to zero, we get the defining equations for $Q$. In particular, we have $I^tQI=Q$, which allows us to rewrite the $a$-coefficient as
$N^tQI+I^tQN=N^tQI-QIN=N^tQI+QNI=(N^tQ+QN)I=0$, that is, $N^tQ+QN=0$.
Similarly, $b$-coefficient gives the same equality, $a^2$- and $b^2$-coefficients give
the same equality $N^tQN=0$, and $ab$-coefficient gives $N^tQIN+(IN)^tQN=N^t(QI+I^tQ)N=0$,
which holds automatically because we already  have $QI+I^tQ=0$. 
%(or, because $N^tQIN+(IN)^tQN=-N^tQNI+N^tQNI=0+0=0$).

Next, let \[Q=
\left(\begin{array}{ccc}
A & B & C\\
-B^t & D & E \\
-C^t & -E^t & F
\end{array}\right),A^t=-A,D^t=-D,F^t=-F,
\]
where, as earlier, the blocks correspond to the decomposition $V_\RR=Im\,N\oplus W \oplus U$.
Let us start with the equality $QN+N^tQ=0$. Rewriting it as $QN=-N^tQ$
and evaluating both sides we get $QN=$
\[
\left(\begin{array}{ccc}
A & B & C\\
-B^t & D & E \\
-C^t & -E^t & F
\end{array}\right)
\left(\begin{array}{c|c|c}
\mathbb{0}_{2k}  & \mathbb{0}_{2k\times 2l} & \mathbb{1}_{2k} \\
\hline
\mathbb{0}_{2l \times 2k} & \mathbb{0}_{2l} &\mathbb{0}_{2l\times 2k}\\
\hline
\mathbb{0}_{2k} & \mathbb{0}_{2k\times 2l} & \mathbb{0}_{2k}
\end{array}\right)=
\left(\begin{array}{c|c|c}
\mathbb{0}_{2k}  & \mathbb{0}_{2k\times 2l} & A \\
\hline
\mathbb{0}_{2l \times 2k} & \mathbb{0}_{2l} & -B^t\\
\hline
\mathbb{0}_{2k} & \mathbb{0}_{2k\times 2l} & -C^t
\end{array}\right)
\]
and
$N^tQ=$
\[
\left(\begin{array}{c|c|c}
\mathbb{0}_{2k}  & \mathbb{0}_{2k\times 2l} & \mathbb{0}_{2k} \\
\hline
\mathbb{0}_{2l \times 2k} & \mathbb{0}_{2l} &\mathbb{0}_{2l\times 2k}\\
\hline
\mathbb{1}_{2k} & \mathbb{0}_{2k\times 2l} & \mathbb{0}_{2k}
\end{array}\right)
\left(\begin{array}{ccc}
A & B & C\\
-B^t & D & E \\
-C^t & -E^t & F
\end{array}\right)
=
\left(\begin{array}{c|c|c}
\mathbb{0}_{2k}  & \mathbb{0}_{2k\times 2l} & \mathbb{0}_{2k} \\
\hline
\mathbb{0}_{2l \times 2k} & \mathbb{0}_{2l} &\mathbb{0}_{2l\times 2k}\\
\hline
A & B & C
\end{array}\right),
\]
so that $QN=-N^tQ$ holds precisely when $A=0,B=0,C=C^t$.
For such $Q$ we automatically have $N^tQN=0$, so that 
what is left to check is $I^tQI=Q$, or, equivalently, $QI=IQ$.
Denoting the three blocks of $I$ on the diagonal by $I_1,I_2,-I_1$,
we get 
\begin{comment}
\[ 
\left(\begin{array}{c|c|c}
I_1 & \mathbb{0}_{2k\times 2l} & \mathbb{0}_{2k}\\
\hline
\mathbb{0}_{2l \times 2k} & I_2 &\mathbb{0}_{2l\times 2k}\\
\hline
\mathbb{0}_{2k} & \mathbb{0}_{2k\times 2l} & -I_1
\end{array}\right)
\]
\end{comment}

\[
IQ=
\left(\begin{array}{ccc}
\mathbb{0}_{2k} & \mathbb{0}_{2k\times 2l} & I_1C\\
\mathbb{0}_{2l\times 2k} & I_2D & I_2E \\
I_1C & I_1E^t & -I_1F
\end{array}\right),
QI=
\left(\begin{array}{ccc}
\mathbb{0}_{2k} & \mathbb{0}_{2k\times 2l} & -CI_1\\
\mathbb{0}_{2l\times 2k} & DI_2 & -EI_1 \\
-CI_1 & -E^tI_2 & -FI_1
\end{array}\right),
\]
so that $C$ is a symmetric $2k\times 2k$-matrix anticommuting with $I_1$,
$D$ is a skew-symmetric $2l\times 2l$-matrix commuting with $I_2$,
$F$ is a skew-symmetric $2k\times 2k$-matrix commuting with $I_1$
and $E$ is a $2l\times 2k$-matrix satisfying $I_2E=-EI_1$.
The vector spaces of such $C,D$ and $F$ have the respective dimensions
$k^2+k,l^2$ and $k^2$.
Writing $E=
\left(
\begin{array}{cc}
E_1 & E_2\\
E_3 & E_4
\end{array}
\right),$
where all blocks have the same size $l\times k$, the
equality $I_2E=-EI_1$ gives $E_4=E_1, E_3=-E_2$,
so that the dimension of such $E$'s is equal $2kl$. We recall that $4k+2l=4n$
so that $l=2(n-k)$.
Now the total dimension of the space of such matrices $Q$
is $k^2+k+(k^2+l^2+2kl)=k^2+k+(k+l)^2=k^2+k+(2n-k)^2$. This is the dimension of the subspace $Hdg_{S(I,N)}$ of
classes in $H^{1,1}(V_\RR/\Gamma,\RR)$ which stay of type $(1,1)$
along $S(I,N)$. 

For tori $V_\RR/\Gamma$ with periods $\lambda$ in (either of the connected components of) $S(I,N)$,  
the form $Q$ determines alternating $(1,1)$-forms $\omega_\lambda(\cdot,\cdot)$ on $(V_\RR,\lambda)$, and the hermitian forms $h_\lambda(\cdot,\cdot)=\omega_\lambda(\cdot,\lambda\cdot)-i\omega_\lambda(\cdot,\cdot)$, associated to forms $\omega_\lambda(\cdot,\cdot)$
(for the definitions see, for example, \cite[Ch. 3.1]{Voisin}), are all identically zero on 
%the complex subspace $V_N\subset V$, corresponding to 
 the (nonzero) $\lambda$-invariant subspace $Im\,N \subset V_\RR$, $\lambda \in S(I,N)$. 
%under isomorphisms $(V,i)\cong (V_\RR,\lambda)$, $\lambda \in S(I,N)$.
%
Thus the hermitian forms $h_\lambda$, all sharing a nonzero isotropic subspace,  are  never positively definite,
so that, finally, there are no K\"ahler classes surviving along any of the connected components of $S(I,N)$.

\begin{comment}
As the restriction of any such $Q$ to the (nonzero) $I$-invariant subspace $Im\,N$
is identically zero, we have that the hermitian forms determined by our $Q$ and representing
(1,1)-classes on tori with periods in (either of the connected components of) $S(I,N)$ all have isotropic vectors, and thus cannot be positively definite,
so that there are no K\"ahler classes surviving along any of the connected components of $S(I,N)$.
\end{comment}
%\end{proof}

\section{Proof of Theorem \ref{Connectivity-theorem}}
%Denote by $\HH$  one of the three algebras: the algebra of quaternions $\langle I,J|I^2=-1,J^2=-1, IJ=-JI\rangle$, the algebra of split quaternions $\langle I,J|I^2=-1,J^2=-1, IJ=-JI\rangle$
%and the algebra $\langle I,J|I^2=-1,J^2=0, IJ=-JI\rangle$. 
%Set $K=IJ$ in all these cases  and set $G_\HH=G_I\cap G_J=G_I\cap G_J\cap G_K$.
Let $T$ be an element of $End\, V_\RR$, set  $G_{T}$ to be the adjoint action stabilizer of $T$ in $G=GL(V_\RR)$,
$G_T=\{g\in G\,|\,gTg^{-1}=g\}$. For $\HH(\varepsilon) \subset End\,V_\RR$ we set 
$G_{\HH(\varepsilon)}\subset G$ to be the adjoint action pointwise stabilizer of $\HH(\varepsilon)$.

%The proof of the twistor path connectivity follows parallel to the proof given in
% \cite{Twistor-lines}. 
We prove Theorem \ref{Connectivity-theorem} here  using the general form
of transversality formulated in Proposition \ref{Proposition-general-transversality} below,
analogously to how the $\HH(-1)$-connectivity was proved in \cite{Twistor-lines}.
On our way to Proposition \ref{Proposition-general-transversality} we will need
the following lemma.

\begin{comment}
In fact,  to prove Theorem \ref{Connectivity-theorem}, as formulated, we only need
a particular form of the transversality, which is stated in Lemma \ref{Transversality-lemma}. 
Our intention to prove Theorem \ref{Connectivity-theorem} using the general form
of transversality, comes from our interest to see how many of the results
of \cite{Twistor-lines}, regarding lines of type $\HH(-1)$ could be (almost directly)
translated into results about lines of type $\HH(1)$. 
\end{comment}

% which based
%refers to Proposition \ref{Proposition-general-transversality} below, but in fact, for the proof of the twistor
%path connectivity  requires only Lemma .
%whose proof requires the following lemma.

\begin{comment}
First we formulate the following general lemma, a particular case of which for the case 
of twistor lines of type $\HH(-1)$ was proved in \cite{Twistor-lines}, stating that
an arbitrary twistor line is uniquely defined by its any two points, non-proportional
as vectors in $End\,V_\RR$.
\end{comment}

%This lemma, like in \cite{Twistor-lines}, will be used to prove 
%Theorem \ref{Connectivity-theorem}.

\begin{lem}
\label{Two-generation}
Let $\varepsilon$ be any of -1,0,1.
The algebra $\HH(\varepsilon)$ is generated, as a real algebra, by any two  linearly independent
 elements from the subspace $\langle i,j,ij \rangle \subset \HH(\varepsilon)$.
\end{lem}

\begin{proof}
%{Proof of Lemma \ref{Two-generation}}
 Set $k=ij$.
Let $i_1,i_2$ be two linearly independent elements in $\langle i,j,k\rangle \subset \HH(\varepsilon)=\langle i,j\,|i^2=-1,j^2=\varepsilon, ij+ji=0\rangle$.
%The elements $i_1,i_2$  are actually contained in the 3-subspace 
%$\langle i,j,k\rangle \subset \HH(\varepsilon)$. 
%  The proof slightly varies for different kinds of $\HH$.
%The proof relies on the use of the 
Let us introduce a bilinear $\HH(\varepsilon)$-valued form $q$ on $\HH(\varepsilon)$, $q(x,y)=xy+yx$. The restriction $Q=q|_{\langle i,j,k\rangle}$ is an $\RR$-valued form.
The elements $x$ and $y$ in $\langle i,j,k\rangle$ anticommute if and only if $x\perp_Q y$.
For $\varepsilon=-1$ and $\varepsilon=1$ the form $Q$ has respective signatures $(0,3)$ and $(2,1)$, so that it is nondegenerate and thus we can construct an 1-dimensional orthogonal complement to the plane 
$\langle i_1,i_2 \rangle$ in $\langle i,j,k \rangle$. If $u,v$ is an orthogonal basis
of $\langle i_1,i_2\rangle$, then $uv$ anticommutes with both $u$ and $v$ and so $uv$ spans the
orthogonal complement $\langle i_1,i_2\rangle^{\perp_Q}$. It is clear now that we have the equality
of vector spaces 
$\HH(\varepsilon)=\langle 1,u,v,uv\rangle=\langle 1, i_1,i_2,i_1i_2\rangle$, which proves the lemma in the cases $\varepsilon=\pm 1$.
%\begin{comment}

In the case of $\HH(0)$ we cannot apply the same argument as $Q=q|_{\langle i,j,k\rangle}$
has now signature $(0,1,2)$, and so we do not apriori know if $uv$, as above, is not in
$\langle i_1,i_2\rangle$ or, what is the same,  $i_1i_2 \notin\langle i_1,i_2\rangle$.
%$Q|_\langle i_1,i_2\rangle^\perp$ must be of dimension 2
In fact, we can give a direct argument here: multiplying by $-1$, if needed, we may assume that
$i_1=i+a_1j+a_2k,i_2=i+b_1j+b_2k$. Let us denote $i^{\prime}=i_1,
j^{\prime}=i_1-i_2$, then the elements $1,i^{\prime}, j^{\prime}, i^{\prime}j^{\prime}$ 
are easily seen to be linearly independent, so that $\HH(0)=\langle 1,i^{\prime}, j^{\prime}, i^{\prime}j^{\prime}\rangle=\langle 1,i_1,i_2,i_1i_2\rangle$.
%$G_{I_1}\cap G_{I_2}=G_{\langle I^{\prime},N^{\prime}, I^{\prime}N^{\prime}\rangle}=G_\HH$.
%\end{comment}
\end{proof}

As a consequence of Lemma \ref{Two-generation} we get the following generalization of an analogous result in \cite{Twistor-lines}
stating that a twistor line is uniquely defined by any pair of its  non-proportional points.

\begin{cor}
Let $S_1,S_2\subset Compl \subset End\,V_\RR$ be any two  generalized twistor lines. 
If the intersection  $S_1\cap S_2$ contains points that are linearly independent as vectors in 
$End\,V_\RR$, 
then $S_1=S_2$. In particular, any two distinct  generalized twistor lines $S_1,S_2\subset Compl$
are either disjoint or $S_1\cap S_2$ consists of a pair of antipodal points $\pm I$. 
\end{cor}

Indeed, let $S_1 \subset \HH(\varepsilon_1) \subset End\,V_\RR,S_2 \subset \HH(\varepsilon_2) \subset End\,V_\RR$ be the twistor lines of the respective types. Then the imaginary units in $S_1\cap S_2$ must be contained in the respective
subspaces $\langle i,j,k\rangle$ of each of the two algebras, so by Lemma \ref{Two-generation} they generate each of $\HH(\varepsilon_1),\HH(\varepsilon_2)$. Hence the latter
subalgebras coincide, therefore $S_1=S_2$. 

Next, as the 
intersection of two distinct generalized twistor lines $S_1$, $S_2$ may only consist
of complex structures that are all proportional to each other, and every 
generalized twistor line in $Compl$ is invariant under the antipodal map $I \mapsto -I$,
we get the rest of the statement of the corollary. 

Another direct consequence of Lemma \ref{Two-generation} is the following corollary. 
\begin{cor}
\label{Corollary-centralizer}
Let $I_1,I_2$ be any two linearly independent imaginary units in
% belonging to the same 
%generalized twistor line of a type 
$\HH(\varepsilon) \subset End\,V_\RR$, where $\varepsilon$ is any of $-1,0,1$.
Then $G_{I_1}\cap G_{I_2}=G_{\HH(\varepsilon)}$.
\end{cor}

Let $I,J \in End\,V_\RR$ satisfy  $I^2=-Id,J^2=\varepsilon Id$ for $\varepsilon$ equal 
$1$ or $-1$, $IJ+JI=0$. Set $K=IJ$. 
One easily calculates  that $\dim_\RR\,G_{I}=\dim_\RR\, G_J=\dim_\RR\,G_K=8n^2$ 
and $\dim_\RR\, G_{\HH(\varepsilon)}=4n^2$, where $\HH(\varepsilon)\subset End\,V_\RR$
is the subalgebra generated by $I$ and $J$. 

%, this is classically known for $\varepsilon=1$
%and can be easily seen from the representation theory for $\HH(1)$ established in the proof
%of Theorem  \ref{Kahler-theorem}.

\begin{lem}
\label{Transversality-lemma}
%The tangent spaces $T_eG_I$, $T_eG_R$, $T_eG_{IR}$ intersect transversally, as a triple,
%along their common subspace $T_eG_\HH$, that is the sum of the subspaces
%$T_eG_I/T_eG_\HH$, $T_eG_R/T_eG_\HH$, $T_eG_{IR}/T_eG_\HH$ in 
Let $\varepsilon$ be any of $1,-1$. We have the direct sum decomposition
$T_eG/T_eG_{\HH(\varepsilon)}=T_eG_I/T_eG_{\HH(\varepsilon)} \oplus T_eG_J/T_eG_{\HH(\varepsilon)} \oplus T_eG_{K}/T_eG_{\HH(\varepsilon)}$.
%
%In the case of $\HH(0)$ the intersection $T_eG_J \cap T_eG_K$ is strictly larger than $T_eG_\HH$,
%and we do not have the direct sum decomposition in this case.
\end{lem}
\begin{comment}
\begin{rem}
It is possible to see that we do not have the direct sum decomposition for $\HH(0)$,
that is, the triple intersection of $G_I/G_{\HH(\varepsilon)},G_N/G_{\HH(\varepsilon)},G_{IN}/G_{\HH(\varepsilon)}$ at $eG_{\HH(\varepsilon)} \in G/G_{\HH(\varepsilon)}$
 is not transversal.
\end{rem}
\end{comment}
\begin{proof}
First of all, the dimension of each of the three quotient spaces in the direct sum is equal $8n^2-4n^2=4n^2$, so that the dimensions sum up to $12n^2=\dim_\RR T_eG/T_eG_{\HH(\varepsilon)}$.

Let us show that the sum of the subspaces is indeed direct.
Let $u\in T_eG_I,v\in T_eG_J, w\in T_eG_{K}$ and
suppose $u+v+w=0\,(mod\,T_eG_{\HH(\varepsilon)})$. Then applying the Lie bracket $[\cdot,I]$
to this equality and using that $uI=Iu$ we get $[v,I]+[w,I]\in T_eG_{\HH(\varepsilon)}$. 
Thus the sum $[v,I]+[w,I]$ must commute with $I$. By construction it also anticommutes with $I$,
so that indeed $[v,I]+[w,I]=0\in T_eG$.
Next, 
$[v,I]$ anticommutes with $I$ and, as $v$ commutes with $J$, $[v,I]$ anticommutes with $J$,
so that finally $[v,I]$ commutes with $K=IJ$. 

In a similar way we show that $[w,I]$ commutes
with $J$, so that $[v,I]=-[w,I]$ commutes with both $K$ and $J$.  
Now if $J^2 =\pm Id$, that is $J$ is an invertible element of our algebra, we have that
$[v,I]$ and $[w,I]$ commute with $I=-\varepsilon JK$, so that $[v,I]=[w,I]=0\in T_eG$
and $v\in T_eG_I\cap T_eG_J=T_eG_{\HH(\varepsilon)}$, $w \in T_eG_K \cap T_eG_I=T_eG_{\HH(\varepsilon)}$.
%Similarly one can show in this case that 
 Then $u\in T_eG_{\HH(\varepsilon)}$ as well. This proves that we have the 
stated direct sum decomposition.
\end{proof}

\begin{prop}
\label{Proposition-general-transversality}
Let $I_1,I_2,I_3$ be complex structures belonging to the same twistor sphere $S$
of type $\HH(\varepsilon)$, where $\varepsilon=\pm 1$.
The submanifolds $G_{I_1}/G_{\HH(\varepsilon)}, G_{I_2}/G_{\HH(\varepsilon)}, G_{I_3}/G_{\HH(\varepsilon)}$ in $G/G_{\HH(\varepsilon)}$
intersect transversally (as a triple) at $eG_{\HH(\varepsilon)}$ if and only if $I_1,I_2,I_3$ are linearly independent as vectors in 
$End\, V_\RR$.
\end{prop}
\begin{proof}
 Let $S=S(I,J)$, where $I^2=-Id,J^2=\varepsilon Id, IJ=-JI$. 
The proof essentially uses the transversality proved in Lemma \ref{Transversality-lemma}, that is, 
the fact that
\begin{equation}
\label{I-R-transversality}
T_eG/T_eG_{\HH(\varepsilon)}=V_I \oplus V_J \oplus  V_{K},
\end{equation}
where we set  $V_I=T_eG_{I}/T_eG_{\HH(\varepsilon)}, V_J=T_eG_{J}/T_eG_{\HH(\varepsilon)}, V_{K}=T_eG_{K}/T_eG_{\HH(\varepsilon)}$, $K=IJ$.
%
%Let us prove that $T_eG/T_eG_{\HH(\varepsilon)}$ decomposes into the direct sum of its subspaces 
We set also $V_j=T_eG_{I_j}/T_eG_{\HH(\varepsilon)}, j=1,2,3$.  
In order to prove the proposition we need to show that $T_eG/T_eG_{\HH(\varepsilon)}=V_1\oplus V_2\oplus V_3$.

We clearly have that $\dim_\RR\,V_j=\dim_\RR\,T_eG_{I_j}-\dim_\RR\,T_eG_{\HH(\varepsilon)}=4n^2, j=1,2,3,$ so that we have the (formal) equality of the dimensions 
 $$\dim_\RR\,T_eG/T_eG_{\HH(\varepsilon)}=\dim_\RR\,V_1+\dim_\RR\,V_2+\dim_\RR\,V_3.$$
Let us now show that the sum $V_1+V_2+V_3$ is, in fact, direct. 

 Imaginary units in $\HH(\varepsilon)$
are all contained in the subspace $\langle I,J,K\rangle\subset \HH(\varepsilon)$, so we can express
 $I_j=a_jI+b_jJ+c_jK, a_j,b_j,c_j \in \RR, j=1,2,3$. 
Suppose that for certain vectors 
$X \in V_1, Y \in V_2$ and $Z \in V_3$
we have $X+Y+Z=0$. Let us decompose the vector $X$ into the sum of its components 
in the respective subspaces of the decomposition (\ref{I-R-transversality}), $X=X_I+X_J+X_{K}$,
and do similarly for $Y$ and $Z$.  Then for $X$ the commutation relation $[X,I_1]=0$
can be written as 
$$a_1[X_J+X_{K},I]+b_1[X_I+X_{K},J]+c_1[X_I+X_J,K]=0.$$
Noting that in the above expression, for example, the term $[X_J,I]$ anticommutes with both $I,J$, hence commutes with $K=IJ$,
and an analogous commutation holds for other terms as well (here it is important that $J$ and $K$ are invertible elements of our $\HH(\varepsilon)$, so that $I=-\varepsilon JK$), we can decompose the expression on the left side of the above equality 
with respect to (\ref{I-R-transversality}),
$$(b_1[X_{K},J]+c_1[X_J,K])+(a_1[X_{K},I]+c_1[X_I,K])+(a_1[X_J,I]+b_1[X_I,J])=0.$$

By Lemma \ref{Transversality-lemma} we conclude that $b_1[X_{K},J]+c_1[X_J,K]=0, a_1[X_{K},I]+c_1[X_I,K]=0$
and $a_1[X_J,I]+b_1[X_I,J]=0$. Assume, that $a_1 \neq 0$. Then we can use
the last two equalities to express 
\begin{equation}
\label{X-I-reduction}
[X_J,I]=-\frac{b_1}{a_1}[X_I,J], [X_{K},I]=-\frac{c_1}{a_1}[X_I,K].
\end{equation}
Note  that Lemma  \ref{Two-generation} implies, that the linear mappings $[\cdot,J] \colon V_I \rightarrow V_{K}, [\cdot, K] \colon V_I \rightarrow V_J$
and $[\cdot, I] \colon V_J \rightarrow V_{K}$  
 are isomorphisms of the respective vector spaces. 
%here it is important that $J$ and $K$ are invertible elements of $\HH(\pm 1)$.
Let us denote these isomorphisms by $F_J,F_{K},F_I$ respectively. Then equalities (\ref{X-I-reduction})
allow us to write 
$$X_J=-\frac{b_1}{a_1}F^{-1}_I \circ F_J(X_I), X_{K}= -\frac{c_1}{a_1}F^{-1}_I \circ F_{K}(X_I),$$
so that 
%$$X=\left(X_I,-\frac{b_1}{a_1}F^{-1}_I \circ F_J(X_I),  -\frac{c_1}{a_1}F^{-1}_I \circ F_K(X_I)\right).$$
$$X=X_I +\left(-\frac{b_1}{a_1}F^{-1}_I \circ F_J(X_I)\right) +  \left(-\frac{c_1}{a_1}F^{-1}_I \circ F_{K}(X_I)\right).$$
Assuming that $a_2,a_3 \neq 0$ we can get similar expressions for $Y$ and $Z$. Actually, for 
the case $\varepsilon =1$ we  have that, in order for $I_1,I_2,I_3$
to be imaginary units, we automatically have $a_j\neq 0, j=1,2,3$, and
for the case $\varepsilon=-1$ choosing
the quaternionic triple $I,J,K$  appropriately, we may assume that all $a_j, j=1,2,3,$ are nonzero.
Using the above representation of $V_J$- and $V_{K}$-components of $X$ in terms of $X_I$, similarly
for $Y$ and $Z$, we can write the equality $X+Y+Z=0$ component-wise,
 with respect to (\ref{I-R-transversality}), 
%and scaling the equations by
%$a_1, a_2, a_3$ we get the linear system
\[
\left(\begin{array}{ccc}
1 & 1 & 1\\
-\frac{b_1}{a_1} & -\frac{b_2}{a_2} & -\frac{b_3}{a_3}\\
-\frac{c_1}{a_1} & -\frac{c_2}{a_2} & -\frac{c_3}{a_3}
\end{array}\right)
\left(\begin{array}{c}
X_I \\
Y_I\\
Z_I 
\end{array}\right)=
\left(\begin{array}{c}
0 \\
0\\
0 
\end{array}\right).
\]
This equation has a nontrivial solution $(X_I,Y_I,Z_I)$ if and only if the columns of the matrix, thus $I_1,I_2,I_3$, are linearly dependent.
\end{proof}

\begin{proof}[Proof of Theorem \ref{Connectivity-theorem}]
Let $I_1,I_2,I_3$ be complex structures, linearly independent as elements
 in $End\,V_\RR$,  belonging  to the same connected component of 
a line $S(I,J)$, where $I^2=-Id,J^2=\varepsilon Id, IJ=-JI$. Then 
%orthogonalizing $I_1,I_2,I_3
%\in \langle I,J,K \rangle$ with respect to the nondegenerate inner product $Q=q|_{\langle I,J,K \rangle},
%q(x,y)=xy+yx,x,y\in \HH(\varepsilon)$, we get,
by Corollary \ref{Corollary-centralizer} we get that
 $G_{I_1} \cap G_{I_2}=G_I\cap G_J=G_{\HH(\varepsilon)}$.

As in \cite{Twistor-lines}, we define the mapping $$\Phi \colon G_{I_1}\times G_{I_2} \rightarrow Compl,$$
$$(g_1,g_2)\mapsto g_1g_2I_3g_2^{-1}g_1^{-1} \in Compl.$$ Let us first show that near the
identity  element $(e,e)\in G_{I_1} \times G_{I_2}$ the mapping $\Phi$ is a submersion onto a neighborhood of $I_3$ in $Compl$.
The differential $d_{(e,e)}\Phi\colon G_{I_1} \times G_{I_2}\rightarrow Compl$
factors through the quotient map 
$$T_eG_{I_1}\oplus T_eG_{I_2} \rightarrow T_eG_{I_1}/T_eG_{\HH(\varepsilon)}\oplus T_eG_{I_2}/T_eG_{\HH(\varepsilon)},$$
denote the resulting map by $$\widetilde{d_{(e,e)}\Phi} \colon T_eG_{I_1}/T_eG_{\HH(\varepsilon)}\oplus T_eG_{I_2}/T_eG_{\HH(\varepsilon)} \rightarrow T_{I_3}Compl.$$ The dimensions of the domain and the target space of $\widetilde{d_{(e,e)}\Phi}$ are equal, so
in order to show that $\Phi$ is a submersion near $(e,e) \in G_{I_1} \times G_{I_2}$
onto a neighborhood of $I_3$ in $Compl$ it is enough to show that $\widetilde{d_{(e,e)}\Phi}$ is injective.
For $X\in  T_eG_{I_1}/T_eG_{\HH(\varepsilon)}$ and $Y\in  T_eG_{I_2}/T_eG_{\HH(\varepsilon)}$ we have 
$$\widetilde{d_{(e,e)}\Phi}(X+Y)=\left.\frac{d}{dt}\right|_{t=0}e^{tX}e^{tY}I_3e^{-tY}e^{-tX}
=(X+Y)I_3-I_3(X+Y)=$$ $$=[X+Y,I_3]\in T_{I_3} Compl.$$
If we assume $[X+Y,I_3]=0$ then $X+Y\in T_eG_{I_3}/T_eG_{\HH(\varepsilon)}$ and so by
Proposition \ref{Proposition-general-transversality} we have that $X=0$ and $Y=0$. This shows the
required injectivity of $\widetilde{d_{(e,e)}\Phi}$.

We have now that $\Phi(G_{I_1}\times G_{I_2})$ contains an open neighborhood of $I_3$ in $Compl$. 
Locally around $I_3\in Compl$ for every point $I$  we can write
$I=g_1g_2I_3g_2^{-1}g_1^{-1}={}^{g_1g_2} I_3$ for $g_1\in G_{I_1},g_2\in G_{I_2}$. 
Then the  complex structures ${}^{g_1g_2}I_2={}^{g_1} I_2,{}^{g_1g_2} I_3$ span the line 
${}^{g_1g_2} S(I_2,I_3)=S({}^{g_1} I_2,{}^{g_1g_2}  I_3)$, %={}^{g_1} S(I_2,{}^{g_2} I_3)$, 
the complex structures ${}^{g_1}I_1=I_1,{}^{g_1} I_2$ span ${}^{g_1} S(I_1,I_2)=S(I_1,{}^{g_1} I_2)$.
Now the consecutive lines $S=S(I_1,I_3)=S(I_1,I_2),{}^{g_1} S=S(I_1,{}^{g_1} I_2)$ and 
${}^{g_1g_2} S=
S({}^{g_1} I_2,{}^{g_1g_2} I_3)$ form a connected path of three lines joining $I_3$ to $I={}^{g_1g_2} I_3$.
Finally, passing to the global picture, we conclude that
%similarly to how it was done in \cite{Bourbaki}, in every connected component of $Compl$ 
 any two points can be joined by 
a path of connected components of generalized twistor lines (possibly, involving more than three 
such lines).
\end{proof}

\end{document}